\documentclass[a4paper, 10pt]{article}
\usepackage{tocloft}
\usepackage{comment}
\usepackage[margin = 1in]{geometry}
\usepackage{aliascnt}

\usepackage{parskip}
\usepackage{nomencl}

\usepackage{fancyhdr} 

\tocloftpagestyle{fancy}
\usepackage{ucs}
\usepackage{units}
\usepackage[T1]{fontenc}
\usepackage{amsmath}
\usepackage{thmtools}
\usepackage{listings}
\usepackage{amssymb} 
\usepackage{amsthm}
\usepackage{amsfonts}
\usepackage{graphicx}
\usepackage{microtype}
\usepackage{mathdots}
\usepackage{mathtools}
\usepackage[shortlabels]{enumitem}
\usepackage{yhmath}
\usepackage{marginnote}
\usepackage{mathrsfs}
\usepackage{bbm}

\usepackage{tikz}
\usepackage{tikz-cd}

\definecolor{linkgreen}{RGB}{33,100,60}
\definecolor{chilligreen}{RGB}{33,100,60}
\definecolor{chillired}{RGB}{140,00,10} 
\definecolor{chillidarkred}{RGB}{120,5,15}

\usepackage[colorlinks=true, citecolor = chilligreen, urlcolor = chillired, backref=page]{hyperref}

\newcommand{\Addresses}{{
  \bigskip
  \footnotesize

  R.~Quinn, \textsc{Utrecht Geometry Center, Universiteit Utrecht, The Netherlands}\\ \nopagebreak
  \textit{E-mail address}: \texttt{r.quinn@uu.nl}

  Q.~Zhu, \textsc{Max Planck Institute for Mathematics, Bonn, Germany}\\ \nopagebreak
  \textit{E-mail address}: \texttt{qzhu@mpim-bonn.mpg.de}

}}

\newcommand{\tb}{\textcolor{chillired}}

\newcommand{\GL}{\operatorname{GL}}

\newcommand{\B}{\mathrm B}
\newcommand{\C}{\mathscr C}

\newcommand{\Sc}{\mathcal S}

\renewcommand{\O}{\mathscr{O}}

\DeclareMathAlphabet{\mathbbold}{U}{bbold}{m}{n}

\newcommand{\E}{\mathbb E}  
\newcommand{\F}{\mathbb F}
\newcommand{\R}{\mathbb R}
\renewcommand{\S}{\mathbb S}

\newcommand{\Z}{\mathbb{Z}}

\usepackage{bbm}

\newcommand{\tmf}{\operatorname{tmf}}
\newcommand{\ku}{\operatorname{ku}}

\newcommand{\MU}{\operatorname{MU}}
\newcommand{\MUP}{\operatorname{MUP}}

\newcommand{\MUR}{{\MU_{\mathbb{R}}}}
\newcommand{\MUPR}{{\MUP_{\mathbb{R}}}}
\newcommand{\MUG}{{\MU^{(\!(G)\!)}}}

\newcommand{\MUGtwo}{{\MU_{(2)}^{(\!(G)\!)}}}
\newcommand{\RG}{{R^{(\!(G)\!)}}}

\newcommand{\BPG}{{\BP^{(\!(G)\!)}}}
\newcommand{\BPGm}{{\BP^{(\!(G)\!)}}\langle m\rangle}

\newcommand{\CPR}{\mathbb{CP}_{\mathbb{R}}}

\newcommand{\Sp}{\mathrm{Sp}}
\newcommand{\Alg}{\mathrm{Alg}}

\newcommand{\LMod}{\mathrm{LMod}}

\newcommand{\Op}{\mathrm{Op}}

\newcommand{\BP}{\operatorname{BP}}

\newcommand{\BPR}{{\BP_{\mathbb{R}}}}
\newcommand{\BPRn}{{\BP_{\mathbb{R}}\langle n\rangle}}
\newcommand{\BU}{\operatorname{BU}}
\newcommand{\BSU}{\operatorname{BSU}}
\newcommand{\BUR}{{\BU_{\mathbb{R}}}}
\newcommand{\BSUR}{{\BSU_{\mathbb{R}}}}
\newcommand{\BGL}{\operatorname{BGL}}

\newcommand{\RO}{\operatorname{RO}}

\newcommand{\Map}{\operatorname{Map}}

\newcommand{\Pic}{\operatorname{Pic}}
\newcommand{\Th}{\operatorname{Th}}

\newcommand{\Res}{\operatorname{Res}}

\renewcommand{\phi}{\varphi}

\newcommand{\map}{\operatorname{map}}
\newcommand{\Infl}{\operatorname{Infl}}

\newcommand{\gl}{\mathrm{gl}}

\newcommand{\Coind}{\operatorname{Coind}}

\newcommand{\gp}{\mathrm{gp}}

\newcommand{\kuR}{\ku_{\R}}

\newcommand{\SU}{\operatorname{SU}}

\newcommand{\EO}{\mathrm{EO}}

\DeclareMathOperator*{\colim}{colim}

\makeatletter
\DeclareRobustCommand{\myuline}[2][0pt]{%
  \ifmmode
    \uline{\hphantom{#2}\kern-#1}%
    \kern#1%
    \mathllap{\mathpalette\my@cont@{#2}}%
  \else
    \uline{\phantom{#2}\kern-#1}%
    \kern#1%
    \llap{\contour{white}{#2}}%
  \fi
}

\newcommand{\ul}{\myuline}
\newcommand{\ol}{\overline}

\usepackage{stmaryrd}

\newcommand{\MURI}{\MUR[\overline{x}_i:i\in I]}

\DeclareFontFamily{U}{min}{}
\DeclareFontShape{U}{min}{m}{n}{<-> udmj30}{}

\usepackage[normalem]{ulem}
\usepackage{contour}
\usepackage{mathtools} 

\contourlength{0.8pt}

\newcommand{\my@cont@}[2]{\contour{white}{\mbox{$\m@th#1#2$}}}
\makeatother
\newcommand{\THR}{\mathrm{THR}}

\usepackage{fontawesome5}

\newcommand{\notehelper}[3]{\textcolor{#3}{$\blacksquare$}\marginpar{\ifodd\thepage\raggedright\else\raggedleft\fi\color{#3}\tiny \textbf{#2:} #1}}

\usepackage[nameinlink,capitalise]{cleveref}
\renewcommand*{\backref}[1]{}
\renewcommand*{\backrefalt}[4]{%
	\ifcase #1%
	\or (Cited on page~#2.)%
	\else (Cited on pages~#2.)%
	\fi%
}

\theoremstyle{definition}

\newtheorem{definition}{Definition}[section]

\newaliascnt{question}{definition}

\aliascntresetthe{question}

\crefname{question}{Question}{Questions}

\newaliascnt{construction}{definition}
\newtheorem{construction}[construction]{Construction}
\aliascntresetthe{construction}
\crefname{construction}{Construction}{Constructions}

\newaliascnt{observation}{definition}
\newtheorem{observation}[observation]{Observation}
\aliascntresetthe{observation}
\crefname{observation}{Observation}{Observations}

\newaliascnt{conjecture}{definition}

\aliascntresetthe{conjecture}
\crefname{conjecture}{Conjecture}{Conjectures}

\newaliascnt{lemma}{definition}
\newtheorem{lemma}[lemma]{Lemma}
\aliascntresetthe{lemma}
\crefname{lemma}{Lemma}{Lemmas}

\newaliascnt{fact}{definition}

\aliascntresetthe{fact}
\crefname{fact}{Fact}{Facts}

\newaliascnt{recollection}{definition}
\newtheorem{recollection}[recollection]{Recollection}
\aliascntresetthe{recollection}
\crefname{recollection}{recollection}{Recollection}

\newaliascnt{setting}{definition}

\aliascntresetthe{setting}
\crefname{setting}{Setting}{Settings}

\newaliascnt{notation}{definition}
\newtheorem{notation}[notation]{Notation}
\aliascntresetthe{notation}
\crefname{notation}{Notation}{Notations}

\newaliascnt{remark}{definition}
\newtheorem{remark}[remark]{Remark}
\aliascntresetthe{remark}
\crefname{remark}{Remark}{Remarks}

\newaliascnt{corollary}{definition}
\newtheorem{corollary}[corollary]{Corollary}
\aliascntresetthe{corollary}
\crefname{corollary}{Corollary}{Corollaries}

\newaliascnt{theorem}{definition}
\newtheorem{theorem}[theorem]{Theorem}
\aliascntresetthe{theorem}
\crefname{theorem}{Theorem}{Theorems}

\newaliascnt{proposition}{definition}
\newtheorem{proposition}[proposition]{Proposition}
\aliascntresetthe{proposition}
\crefname{proposition}{Proposition}{Propositions}

\newaliascnt{example}{definition}
\newtheorem{example}[example]{Example}
\aliascntresetthe{example}
\crefname{example}{Example}{Examples}

\aliascntresetthe{recollection}
\crefname{recollection}{Recollection}{Recollections}

\newaliascnt{folklore}{definition}
 
\aliascntresetthe{folklore}
\crefname{folklore}{assumption}{assumptions} 
\Crefname{folklore}{Assumption}{Assumptions}

\newaliascnt{assumption}{definition}

\aliascntresetthe{assumption}
\crefname{assumption}{Assumption}{Assumptions}

\newtheorem{mainthm}{Theorem}

\newaliascnt{maincor}{mainthm}
\newtheorem{maincor}[maincor]{Corollary}
\aliascntresetthe{maincor}
\crefname{maincor}{Corollary}{Corollaries}

\newaliascnt{mainex}{mainthm}

\aliascntresetthe{mainex}
\crefname{mainex}{Example}{Examples}

\newaliascnt{mainrem}{mainthm}
\newtheorem{mainrem}[mainrem]{Remark}
\aliascntresetthe{mainrem}
\crefname{mainrem}{Remark}{Remarks}

\newtheorem*{theorem*}{Theorem}
\newtheorem*{example*}{Example}
\newtheorem*{question*}{Question}
\newtheorem*{problem*}{Problem}
\newtheorem*{thmG*}{Theorem G}

\begin{document}

	\title{\vspace{-1cm}\textbf{Structured Quotients in Real Homotopy Theory}}
	\author{\textsc{Ryan Quinn \& Qi Zhu}}
    \date{} 
	\maketitle

            \hypersetup{linkcolor=chilligreen}

		\begin{abstract}
			\noindent We equip quotients of Real bordism $\MUR$ with the structure of a ring with involution, an important source of examples being the truncated Real Brown--Peterson spectra $\BPR \langle n \rangle$. Motivated by this, we orient Lubin--Tate theory by higher truncated Brown--Peterson spectra, which is a key input for Meier--Shi--Zeng's transchromatic isomorphism theorem. We use these orientations to characterize the higher truncated Brown--Peterson spectra that are equivalent to a form of Lubin--Tate theory after chromatic localization.
		\end{abstract}
\hypersetup{linkcolor=chilligreen}

\begin{figure}[ht!]
\centering
\includegraphics[width=130mm]{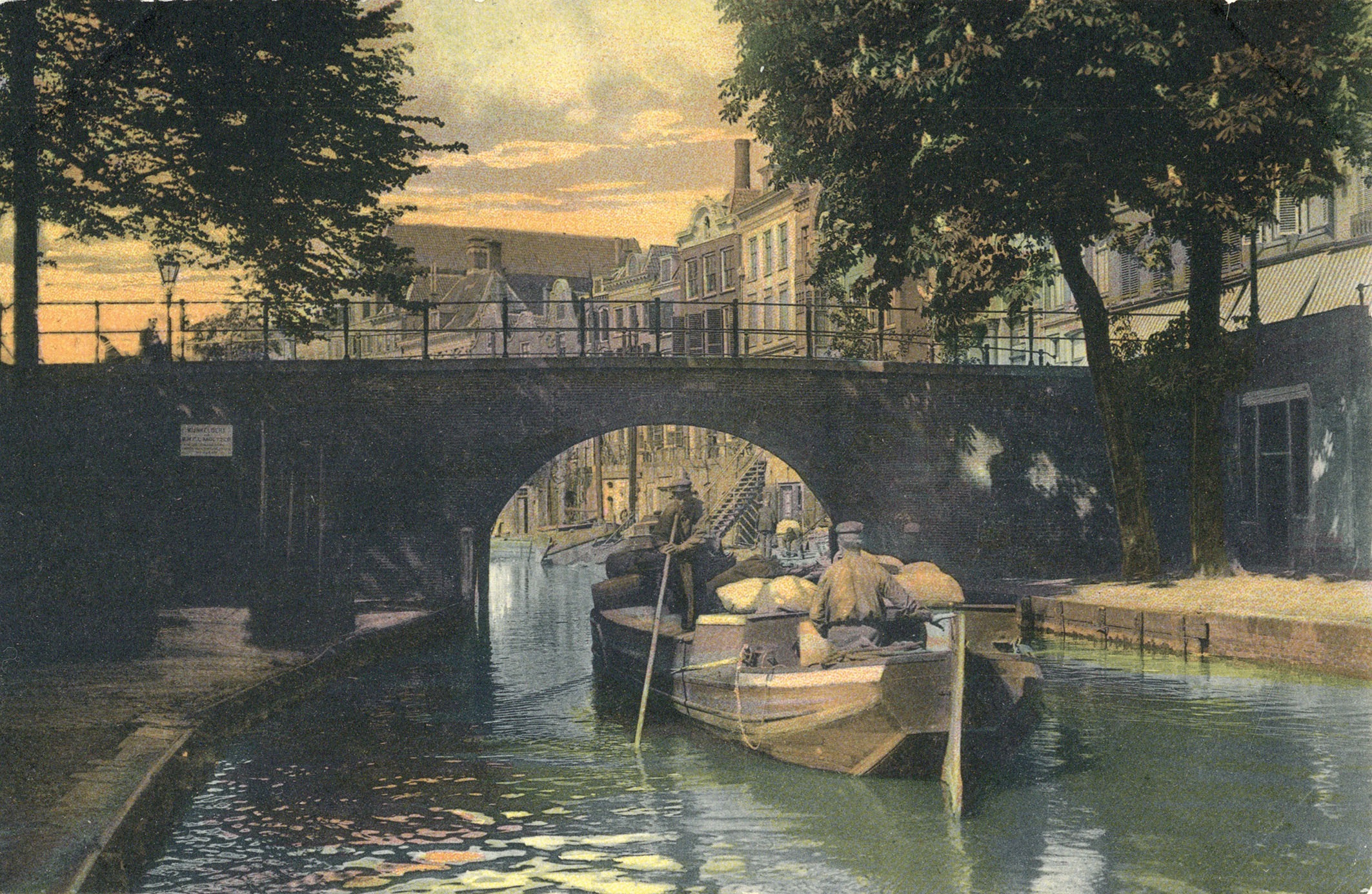}
\end{figure}


\begingroup

\renewcommand\thefootnote{}
\footnote{Gezicht op de Oudegracht te Utrecht -- 1895.}%
\footnote{\emph{Date}: \today}
\addtocounter{footnote}{-2}
\endgroup
\thispagestyle{empty} 
\setcounter{tocdepth}{2}    
\setcounter{secnumdepth}{4} 
\tableofcontents

\thispagestyle{empty}
\newpage
\section{Introduction}
Quotients in higher algebra are much more delicate than their counterparts in classical algebra, e.g.~the Moore spectrum $\S/2$ famously does not even admit a unital multiplication. Nonetheless, there are nowadays many techniques for constructing multiplicative structures on quotients in higher algebra  \cite{robinson1989obstruction, ekmm1997stablehomotopy,stricklandMU,lazarev2001homotopy, bakerJeanneret2002brave, duggerShipley2006postnikov, Angeltveit_2008,hopkinslurieambi,Basu_2017,hahn2018quotientsrings,hahnWilson2022redshift,burklund2022multiplicativestructuresmoorespectra, rognes2025localizationsequenceslogarithmictopological, willumsgaard2026obstructionsassociativitystablehomotopy}. However, the same cannot be said for equivariant higher algebra. Developing such a theory is the main goal of this article, particularly for quotients of the Araki--Landweber Real bordism theory $\MUR$.

\subsubsection*{Structured quotients of Real bordism}
Many $C_2$-spectra of interest can be constructed as quotients from $\MUR$ -- most notably, these include the Real Brown--Peterson theory $\BPR$ as well as its truncated variants $\BPRn$, see \cref{example: BPRJ}. These examples generalize familiar $C_2$-spectra of low chromatic heights: for $n = 1$ it specializes to Atiyah's $K$-theory with Reality $\ku_{\R(2)} \simeq \BPR \langle 1 \rangle$ and for $n = 2$ to $\tmf_1(3)_{(2)} \simeq \BPR \langle 2 \rangle$ by \cite{hillmeier2017}.\footnote{A $C_2$-action on $\tmf_1(3)$ is inherited through an algebro-geometrically defined action on the compactified moduli stack $\ol{\mathcal{M}}_1(3)$ of elliptic curves with a chosen point of order $3$, see \cite{hillmeier2017}.} Moreover, one obtains higher variants $\BPG$ and $\BPGm$ for a finite group $G \geq C_2$ through certain norm constructions (\cref{example: BPGJ}). While already intrinsically interesting objects, these can be used to construct various useful non-equivariant spectra such as connective models of higher real $K$-theories \cite{beaudryHillShiZeng2021modelsLubinTate, carrick2025higherrealktheoriesfinite}, an insight from Hill--Hopkins--Ravenel's resolution of the Kervaire invariant one problem \cite{HHR16}.

Our first main result concerns the structure on such objects and produces the structure of a $C_2$-ring spectrum with involution on all of these quotients, modelled by $\E_{\sigma}$-algebras for the $C_2$-sign representation $\sigma$. This is particularly relevant in the study of Real algebraic $K$-theory, which can be approached through Real trace methods. Indeed, Real topological Hochschild homology \cite{hesselholtMadsen2015Realalgebra, dottoMoiPatchkoriaReeh2021THR} is precisely defined for $\E_{\sigma}$-algebras. As such, Angelini-Knoll--Kong--Quigley compute the Real syntomic cohomology of \(\BPRn\) for \(-1\leq n\leq 2\), using that $C_2$-commutative models of these are available, see \cite[Theorem D]{gabe2025realsyntomiccohomology}.
They state that the only obstruction to an extension to all \(n\) is the lack of \(\E_\sigma\)-algebra structures on \(\BPRn\) for \(n>2\), see \cite[p.~4]{gabe2025realsyntomiccohomology}. Our theorem produces this structure, therefore enabling a study of $\THR(\BPRn)$ and $\THR(\BPRn/\MUR)$ as well as Angelini-Knoll--Kong--Quigley's program for $\BPRn$.

Let us now state a precise version of our first main result. Throughout, let $G = C_{2^n}$ for some $n$ and let $\rho$ be the regular $C_2$-representation. For every $G$-representation $V$ there exists an equivariant little disk operad $\E_V$ controlling equivariant multiplicative structures \cite{blumberghill2015operadic, horev2019genuineequivariantfactorizationhomology, Hill22disks}. Moreover, when norming up an $\E_V$-algebra $R$ along $G \leq G'$, then $N_G^{G'} R$ naturally inherits the structure of a so-called $\Coind_{G}^{G'} \E_V$-algebra, where $\Coind_G^{G'} \colon \Op_{G, \infty} \to \Op_{G', \infty}$ is the right adjoint to the restriction functor \cite[Construction 2.2.5]{quinnZhu2026multiplicativeequivariantthomspectra}. This applies in particular to examples like $\MUG \coloneqq N_{C_2}^G \MUR$ and $\BPG \coloneqq N_{C_2}^G \BPR$.
\begin{mainthm}[\cref{thm:MUR_quotients}, \cref{thm:MUG_quotients}]\label{main:MUR_quotients}
    Let $I \subseteq \Z_{\geq 1}$ be an indexing set.
    \begin{enumerate}[(i)]
        \item Let \(\{\overline{x}_{i}\}_{i\in I}\) be a collection of elements with \(\overline{x}_i\in \pi^{C_2}_{i\rho}\MUR\).
        Then \(\MUR/(\overline{x}_i:i\in I)\) admits an \(\E_\sigma\)-\(\MUR\)-algebra structure.
        \item Let \(\{\overline{x}_{i}\}_{i\in I}\) be a collection of elements with \(\overline{x}_i\in \pi^{C_2}_{i\rho}\MUG\).
        Then \(\MUG/(G\cdot \overline{x}_i:i\in I)\) admits a \(\Coind^G_{C_2}\E_\sigma\)-\(\MUG\)-algebra structure.
    \end{enumerate}
\end{mainthm}

\begin{mainrem} \label{main: remark generic}
    In general, an $\E_{\sigma}$-algebra structure on a $C_2$-spectrum $R$ provides an $\E_1$-algebra structure on $\Res_e^{C_2} R$ as well as an $N_e^{C_2} \Res_e^{C_2} R$-$N_e^{C_2} \Res_e^{C_2}R$-bimodule structure on $R$, see \cite{horev2019genuineequivariantfactorizationhomology, Hill22disks}. On the other hand, it does not provide a unital equivariant multiplication on $R$: there is no underlying $\E_1$-algebra structure in $C_2$-spectra coming from an $\E_{\sigma}$-algebra structure.
    
    So \cref{main:MUR_quotients} enhances the associative multiplications on quotients of $\MU$ that were established by Angeltveit, Basu--Sagave--Schlichtkrull and Hahn--Wilson  \cite{Angeltveit_2008, Basu_2017, hahnWilson2022redshift}. However, it does not enhance the multiplications to unital equivariant multiplications on quotients of $\MUR$. Indeed, this is provably impossible, since $\E_{\sigma}$ is generically the best possible structure on such quotients. An argument by Shi, recorded by Bachmann--Hahn, shows that $\BPR/\ol{v}_i$ does not admit the structure of a homotopy associative $C_2$-ring spectrum (\cref{remark: MUR quotient theorem}). This rules out any $\E_1$-version of \cref{main:MUR_quotients}. Moreover, there are many quotients whose underlying spectra do not admit $\E_2$-algebra structures, so also $\E_{2\sigma}$ cannot be achieved generically, making $\E_{\sigma}$ the best possible $C_2$-little disk operad one can hope for. Instead, our result equips those quotients with bimodule structures over the norm of their underlying. That is precisely the structure needed to construct $\THR(R) \coloneqq R \otimes_{N_e^{C_2} \Res_e^{C_2} R} R$, whence our interest in $\E_{\sigma}$-algebra structures.
\end{mainrem}

While Hahn--Shi \cite{hahnRealOrientationsLubin2020} have constructed an $\E_{\sigma}$-algebra structure on quotients of the periodic Real bordism spectrum $\MUPR$, their approach does not seem feasible for the non-periodic version. Suppose we want to quotient out certain homogeneous elements in $\pi_{*\rho} \MUPR$. Then, Hahn--Shi crucially use the periodicity generator to move every such element to the same degree. This degree can then be accessed through the geometry of $\mathbb{HP}^{\infty}$, see \cite[Proposition 4.2]{hahnRealOrientationsLubin2020}. Instead, we are inspired by Basu--Sagave--Schlichtkrull's work on writing quotients of $\MU$ as $\E_1$-Thom spectra over $\SU$, see \cite{Basu_2017}. In fact, we prove a $C_2$-equivariant version of their result, although our methods are completely different and involve certain lifting techniques from \cite{quinnZhu2026multiplicativeequivariantthomspectra}.

One could try to naively equivariantize Basu--Sagave--Schlichtkrull's argument, and the first hurdle is already realizing that the appropriate $C_2$-refinement for $\SU$ is not $\SU_{\R}$, i.e.~$\SU$ with the complex conjugation action, but rather $\B \Omega^{\sigma} \SU_{\R}$. It turns out that this non-trivial twist complicates attempts of a naive generalization, see \cref{remark: BSS hard}. Nonetheless, using our methods we manage to prove that quotients of $\MUR$ can be obtained as $C_2$-Thom spectra over $\B \Omega^{\sigma} \SU_{\R}$ through an $\E_{\sigma}$-algebra map (\cref{corollary: Real BSS}).

As the main examples for our \cref{main:MUR_quotients}, we immediately obtain structured versions of the Real truncated Brown--Peterson spectra, recalled in \cref{example: BPRJ} and \cref{example: BPGJ}, answering a question of Angelini-Knoll--Kong--Quigley \cite[Remark 5.7]{gabe2025realsyntomiccohomology}:

\begin{maincor}\label{main:BPRnkRn}
    The truncated Real Brown--Peterson spectra \(\BPRn\) admit \(\E_\sigma\)-\(\MUR\)-algebra structures and higher truncated Brown--Peterson spectra \(\BPG \langle m \rangle\) admit \(\Coind^{G}_{C_2}\E_\sigma\)-\(\MUG\)-algebra structures for all $n \geq 1$.
\end{maincor}

The theorem also immediately applies to certain generalized higher truncated Brown--Peterson spectra, giving them a $\Coind_{C_2}^G \E_{\sigma}$-algebra structure, and thereby answering a question by Hahn--Wilson \cite[Problem 5.5]{AimPL_equivstable_5}. We will recall the construction of these objects in \cref{example: BPGJ}.

While an $\E_{\sigma}$-algebra structure on quotients of $\MUR$ is optimal generically (\cref{main: remark generic}), an enhancement to an $\E_{1+2\sigma}$-algebra structure is expected in this specific case of truncated Real Brown--Peterson spectra \cite[Remark 1.0.14]{hahnWilson2022redshift}. This would in particular equip $\BPRn$ and $\BPGm$ with an equivariant associative multiplication. In upcoming joint work with Carrick--Hill--Stewart, we set up an equivariant version of $\mathrm{TAQ}$ and compute equivariant centers to establish this result.

\subsubsection*{Orientations by twisted monoid quotients}
Combining the interest of quotients of Real bordism and orientation theory leads to the study of orientations by the twisted monoid quotients $\BPGm$. Building on work of Hahn--Shi \cite{hahnRealOrientationsLubin2020}, Beaudry--Hill--Shi--Zeng \cite{beaudryHillShiZeng2021modelsLubinTate} produced certain explicit $G$-equivariant homotopy ring maps
\[ \BPG = N_{C_2}^G \BPR \longrightarrow E(k, \Gamma_h) \]
to Lubin--Tate theory associated to a perfect field $k$ and a formal group law $\Gamma_h$ of height $h$, where the $G$-action on $E(k, \Gamma_h)$ comes from the Morava stabilizer group action. This allows us to transport known features about $\BPG$ over to $E(k,\Gamma_h)$ and is in particular a technique to understand the inexplicit group action on $E$-theory obtained through obstruction theory.

We produce a factorization of this map through the higher truncated Brown--Peterson spectra $\BPGm$, whose definition we recall in \cref{example: BPGJ}.

\begin{mainthm}[\cref{thm:BHSZ_refinement}] \label{mainthm: BHSZ}
    The Beaudry--Hill--Shi--Zeng orientation \(\BPG \to E(k,\Gamma_h)\) refines to a \(\Coind^G_{C_2}\E_\rho\)-ring map.
    Furthermore, there are \(G\)-equivariant factorizations
    \[
    \begin{tikzcd}
        \BPG \ar[r]\ar[d] & E(k,\Gamma_h)
        &
        \raisebox{-0.9em}{\text{and}}
        &
        D^{-1}\BPG \ar[r]\ar[d] & E(k,\Gamma_h) \\
        \BPGm \ar[ru, dashed]
        &&
        &
        D^{-1}\BPGm \ar[ru, dashed]
    \end{tikzcd}
    \]
    of $\BPG$-modules. Here, \(D\in \pi^{G}_{*\rho_G} \left(\MUG \right)\) is a certain explicit class considered by Beaudry--Hill--Shi--Zeng \cite[Section 6]{beaudryHillShiZeng2021modelsLubinTate}.
\end{mainthm}

In fact, the existence of this factorization (\cref{mainthm: BHSZ}) was long expected by experts in the field, although a rigorous derivation relies on further equivariant multiplicative structures that were not available. Our previous work on structured Real orientations \cite{quinnZhu2026multiplicativeequivariantthomspectra} gives us just enough structure to perform this, allowing us to apply a general factorization result (\cref{prop: factorization}).

To our knowledge, \cref{mainthm: BHSZ} gives the first construction of the maps
\[ \BPGm \longrightarrow E(k, \Gamma_h). \]
Nonetheless, they have already been used to great effect in the literature, particularly in relation to identifying periodicities of $E(k, \Gamma_h)$. Indeed, these orientations feature in the transchromatic isomorphism theorem from Meier--Shi--Zeng \cite{meier2024transchromaticphenomenaequivariantslice}, which then Duan--Hill--Li--Liu--Shi--Wang--Xu use to determine and recover a myriad of $\RO(G)$-graded periodicities of $E(k,\Gamma_h)$, see
\cite{duan2025periodicityfinitecomplexityhigher}. We give an overview of these applications in the concluding \cref{subsection: periodicity lubin tate}.

Moreover, these orientations are one of the main sources of interest for the higher truncated Brown--Peterson theories: they allow us to access the higher real $K$-theories, as e.g.~shown to great effect in \cite{carrick2025higherrealktheoriesfinite}. For instance, the existence of the maps in \cref{mainthm: BHSZ} along with results from \cite{beaudryHillShiZeng2021modelsLubinTate} allows us to give a new equivariant model for higher real $K$-theory. 

Let us denote by $\ul{\map}_G$ the mapping $G$-spectrum.

\begin{maincor}[\cref{cor:BPGmEh_equivalence}] \label{maincor: equivariant equivalence}
    Let $G = C_{2^n}$ and $h = 2^{n-1}m$. Suppose that $k^{\times}$ contains all $q = (2^m-1)$-roots of unity. There is an equivalence
    \[ \ul{\map}_G \left(\Sigma_+^{\infty} \mathrm{EG}, L_{\Infl_e^G K(h)} D^{-1}\BPGm \right) \xrightarrow{\ \simeq \ } E(k,\Gamma_h)^{hC(k,m)} \]
    of $\BPG$-modules, where $C(k,m)$ is a subgroup of the extended Morava stabilizer group recalled in \cref{recollection: bhsz}.
\end{maincor}

In fact, it was expected by experts that this equivalence already holds true before inverting $D$, inspired by Baker--Würgler's non-equivariant equivalence $L_{K(h)} \BP\langle h \rangle \simeq \widehat{E}(h)$ to completed Johnson--Wilson theory \cite[Theorem 4.1]{bakerWuergler1989liftings}. As such, we show that $E(k,\Gamma_h)^{hC(k,m)}$ is an equivariant refinement of completed Johnson--Wilson theory (\cref{remark: Baker wuergler}).

While we have managed to verify that the map in \cref{maincor: equivariant equivalence} remains an equivalence for $n = 1$ as well as for $(m,n) = (1,2)$ before inverting $D$, we were surprised to find that it is not an equivalence in general. We give a full characterization of when
\[ \ul{\map}_G \left(\Sigma_+^{\infty} \mathrm{EG}, L_{\Infl_e^G K(h)} \BPGm \right) \longrightarrow E(k,\Gamma_h)^{hC(k,m)} \]
is an equivalence of $\BPG$-modules.

\begin{mainthm}[\cref{prop: positive examples lubin tate}, \cref{theorem: not higher real k theory}] \label{mainthm: characterization}
    Let $G = C_{2^n}$ and $h = 2^{n-1}m$. Suppose that $k^{\times}$ contains all $q = (2^m-1)$-roots of unity. The map
    \[ \ul{\map}_G \left(\Sigma_+^{\infty} \mathrm{EG}, L_{\Infl_e^G K(h)} \BPGm \right) \longrightarrow E(k,\Gamma_h)^{hC(k,m)} \]
    is an equivalence of $\BPG$-modules if and only if $n = 1$ or $(m,n) = (1,2)$.
\end{mainthm}

We prove this through a $\pi_*(L_{K(h)}-)/(2,v_1,\cdots, v_{h-1})$ computation. Here, Beaudry--Hill--Shi--Zeng computed 
\[ \pi_*^e \left(E(k,\Gamma_h)^{hC(k,m)} \right)/(2,v_1, \cdots, v_{h-1}) \cong \F_2[t^{\pm 1}] \] 
in \cite[Proposition 7.1]{beaudryHillShiZeng2021modelsLubinTate}, but the complexity for the analogous computation on the left-hand side of \cref{mainthm: characterization} increases rapidly, so we ultimately distinguish
\[ \pi_* \left(L_{K(h)} \Res_e^G \BPGm \right)/(2,v_1,\cdots, v_{h-1}) \qquad \text{from} \qquad \pi_*^e \left(E(k,\Gamma_h)^{hC(k,m)} \right)/(2,v_1,\cdots, v_{h-1}) \]
by counting the $\F_4$-points on their associated affine schemes (\cref{prop: F4 points computation}).

\subsection*{Outline}
The remainder of the article is organized as follows:

In \cref{section: structured quotients of Real bordism} we produce an $\E_{\sigma}$-$\MUR$-algebra structure on quotients of $\MUR$ through designer spectra and lifting methods. We furthermore discuss higher group versions and produce a $\Coind_{C_2}^G \E_{\sigma}$-$\MUG$-algebra structure on twisted monoid quotients of $\MUG$.

In \cref{section: orientations by twisted monoid quotients} we factor the Beaudry--Hill--Shi--Zeng orientations $\BPG \to E(k, \Gamma_h)$ through the higher truncated Brown--Peterson theory $\BPGm$, which we use to study the $K(h)$-localization of $\BPGm$ and its relation to higher real $K$-theory. We end by discussing known applications of the factorization, namely periodicities of Lubin--Tate theory.

In \cref{section: twisted monoid quotients via parametrized colimits} we record a definition of twisted monoid quotients as a total cofiber of an equivariant cube. This is made precise using parametrized colimits.

\subsection*{Notations \& Conventions}
Our terminology will be compatible with \cite{quinnZhu2026multiplicativeequivariantthomspectra}, and in particular we will freely be using the language of parametrized higher algebra. Nonetheless, we will recall relevant notions in the body of the article, so this paper can be read independently from \cite{quinnZhu2026multiplicativeequivariantthomspectra}. We collect some notational conventions here that we will use throughout the whole article.
\begin{enumerate}[(1)]
    \item Throughout, $G$ will always be a finite group.
    \item We denote by $\rho$  and \(\sigma\) the regular and sign representations of $C_2$ respectively.
    \item Underlined categories are $G$-$\infty$-categories, e.g. $\ul{\Sp}_G$ is the $G$-$\infty$-category which on level $H \leq G$ is $\Sp_H$.
    \item Let $X,Y \in \Sp_G$. We  denote by $\ul{\map}_G(X, Y)$ the mapping $G$-spectrum refining the mapping $G$-space $\ul{\Map}_G(X, Y)$. We also write $\ul{\map}_{\ul{\Sp}_G}(X, Y)$ and $\ul{\Map}_{\ul{\Sp}_G}(X, Y)$.
    \item The $G$-$\infty$-operad $\E_{\infty}^G$ is the terminal $G$-$\infty$-operad. So $\E_{\infty}^G$-algebras admit all norms -- they are also known as normed algebras, ultracommutative algebras, or \(G\)-\(\E_\infty\)-algebras.
    \item There is a restriction functor of equivariant $\infty$-operads $\Res_H^G \colon \Op_{G, \infty} \to \Op_{H, \infty}$ which admits a right adjoint $\Coind_H^G \colon \Op_{H, \infty} \to \Op_{G, \infty}$, the coinduction functor. See \cite[Section 1.3.1]{stewart2025tensorproductsequivariantcommutative} and \cite[Remark 2.2.5]{quinnZhu2026multiplicativeequivariantthomspectra} for a discussion in the language of parametrized higher categories.
\end{enumerate}

\subsection*{Acknowledgements}
We thank Christian Carrick, Markus Hausmann, Kaif Hilman, Ishan Levy, Guchuan Li, Lennart Meier and XiaoLin Danny Shi for helpful and encouraging discussions.
RQ is funded by the NUI Travelling Studentship. QZ is supported by the Max Planck Institute for Mathematics (MPIM) in Bonn and is thankful
for its financial support and for providing  conducive working environments. QZ furthermore thanks Utrecht University for its hospitality, where a portion of this article was written.

\section{Structured quotients of Real bordism} \label{section: structured quotients of Real bordism}

\subsection{Equivariant designer polynomial algebras}\label{sec:designer}

The main construction we need is an \(\MUR\)-version of Hahn--Wilson's designer polynomial \(\MU\)-algebras \cite[Construction 2.6.1]{hahnWilson2022redshift}.

\begin{proposition}\label{construction:designer_MUR_poly}
    The free \(\E_1\)-\(\MUR\)-algebra on a class in degree \(i\rho\), denoted by \(\tb{\MUR[\overline{x}_i]}\), admits an \(\E^{C_2}_\infty\)-\(\MUR\)-algebra structure.
\end{proposition}

\begin{proof}
Recall that the Real \(J\)-homomorphism \(\ul{\Z}\times\BUR \to \ul{\Pic}_{C_2}(\ul{\Sp}^{C_2})\) is an \(\E^{C_2}_\infty\)-map of $C_2$-spaces \cite[Appendix A.2]{quinnZhu2026multiplicativeequivariantthomspectra}. Postcomposing this by $\ul{\Pic}_{C_2}(\ul{\Sp}^{C_2}) \to \ul{\Sp}^{C_2}$ yields a map $\ul{\Z} \times \BUR \to \ul{\Sp}^{C_2}$. Parametrized operadically left Kan extending \cite{nardinshah2022equivarianttopos} along the $\E_{\infty}^{C_2}$-map \( \ul{\Z}\times\BUR \to \ul{\Z}\) gives a lax \(C_2\)-symmetric monoidal functor \(\ul{\Z}\to \ul{\Sp}^{C_2}\).
    Precomposing with \(-\cdot i \colon \ul{\Z} \to \ul{\Z}\),
    restricting and left Kan extending along \(\ul{\Z}_{\geq 0}\hookrightarrow \ul{\Z}\) yields a lax $C_2$-symmetric monoidal functor $\ul{\Z} \to \ul{\Sp}^{C_2}$, which we will suggestively call $\MUR[\ol{x}_i]$.
    \begin{center}
        \begin{tikzcd}
            & & \ul{\Z} \times \BUR \arrow[r] \arrow[d] & \ul{\Pic}_{C_2}(\ul{\Sp}^{C_2}) \arrow[r] & \ul{\Sp}^{C_2}
            \\ \ul{\Z}_{\geq 0} \arrow[d] \arrow[r] & \ul{\Z} \arrow[r, "i"] & \ul{\Z} \arrow[urr, dashed]
            \\ \ul{\Z} \arrow[uurrrr, bend right, dotted, "{\MU_{\R}[\ol{x}_i]}", swap]
        \end{tikzcd}
    \end{center}
    Forgetting the grading yields an \(\E^{C_2}_\infty\)-\(\MUR\)-algebra, where the $\MUR$-algebra structure is induced by the $\E_\infty^{C_2}$-map $\BUR \to \ul{\Z} \times \BUR$ that includes into the $0$-component. Let us now argue that this construction is indeed a refinement of the free $\E_1$-$\MUR$-algebra on a generator of degree $i\rho$, thus justifying the notation. 
    
    The underlying $C_2$-spectrum is $\bigoplus_{k \geq 0} \Sigma^{ki\rho} \MUR$ by the left Kan extension formula, which is in particular strongly even. Moreover, the choice of a generator $\ol{x}_i$ in $\pi_{i\rho}(\MUR[\ol{x}_i])$ determines an $\E_1$-map from the free $\E_1$-$\MUR$-algebra on a generator of degree $i\rho$. This is a map between strongly even $C_2$-spectra, which is an equivalence on underlying, checked on homotopy groups.\footnote{On underlying it precisely reduces to Hahn--Wilson's construction \cite[Construction 2.6.1]{hahnWilson2022redshift}.} On the other hand, any map between strongly even $C_2$-spectra, which is an equivalence on underlying, is itself an equivalence \cite[Lemma 3.4]{hillmeier2017}.
\end{proof}

\begin{corollary}\label{lemma:MURI}
    Let \(I\subseteq \Z_{\geq 1}\) and let \(\{\overline{x}_i\}_{i\in I}\) denote a collection of classes in degrees \(i\rho\).
    Then, $\MUR[\overline{x}_i : i\in I]$
    admits an \(\E^{C_2}_\infty\)-\(\MUR\)-algebra structure.
\end{corollary}

\begin{proof}
By taking filtered colimits, we may assume that \(I\) is finite.
By definition, \(\MUR[\overline{x}_i : i\in I]\) is the tensor product over all \(i\in I\) of \(\MUR[\overline{x}_i]\) in \(\myuline{\LMod}_{\MUR}\). On the other hand, each copy of these \(\MUR[\overline{x}_i]\) admits an \(\E^{C_2}_\infty\)-\(\MUR\)-algebra structure by \cref{construction:designer_MUR_poly}.
\end{proof}

Inspired by a result of Hahn--Yuan \cite[Theorem 7.1]{HahnYuan} we formulate a lifting result using techniques from \cite{quinnZhu2026multiplicativeequivariantthomspectra}. The main application we have in mind involves the polynomial $\MUR$-algebra that we have obtained above.
\begin{proposition}\label{prop:lift}
    Let $E$ be a strongly even $\E_{\infty}^{C_2}$-spectrum and suppose that \(\pi^{C_2}_{i\rho}(E)\) is finitely generated for each \(i\in \Z\).
    Any $\E_2$-map $\Sigma_+^{\infty} \BU \to E^e$ lifts uniquely to an $\E_{\rho}$-map $\Sigma_+^{\infty}\BUR \to E$.
\end{proposition}

\begin{proof}
    Consider the following sequence of equivalences of $C_2$-mapping spaces:
    \begin{align*}
        \ul{\Map}_{\ul{\Alg}_{\E_{\rho}}(\ul{\Sp}^{C_2})}(\Sigma_+^{\infty}\BUR, E) &\simeq \ul{\Map}_{\ul{\Alg}_{\E_{\rho}}(\ul{\Sc}^{C_2})}(\BUR, \Omega^{\infty}E)
        \\ &\simeq \ul{\Map}_{\ul{\Alg}_{\E_{\rho}}^{\gp}(\ul{\Sc}^{C_2})}(\BUR, \GL_1 E)
        \\ &\simeq \ul{\Map}_{\ul{\Sc}^{C_2}_*}(\BSU_{\R}, \mathrm{B}^{\rho}\GL_1 E)
        \\ &\simeq \ul{\Map}_{\ul{\Sp}^{C_2}}(\Sigma^{\infty}\BSU_{\R}, \Sigma^{\rho}\gl_1 E).
    \end{align*}
    The third and fourth equivalences use the equivariant recognition theorems, see \cite{Guillou2012EquivariantIL,cnossen2024normedequivariantringspectra, juran2025genuineequivariantrecognitionprinciple}.
    Roytman's verification that the Bredon cohomology of $\BSUR$ with coefficients in $\ul{\Z}$ is strongly even \cite[Theorem 6.1]{Roytman2023}, along with the assumptions we imposed, allows us to apply our cohomological slice tower techniques \cite[Remark 5.3.9]{quinnZhu2026multiplicativeequivariantthomspectra} to obtain the desired liftings.
\end{proof}

\begin{corollary} \label{prop: lift}
    Let \(I\subseteq \Z_{\geq 1}\) and let \(\{\overline{x}_i\}_{i\in I}\) denote a collection of classes in degrees \(i\rho\).
    Then any \(\E_2\)-map \(\Sigma^\infty_+\BU\to \MU[{x}_i : i\in I]\) lifts uniquely to an \(\E_\rho\)-ring map \(\Sigma^\infty_+\BU_{\R}\to \MUR[\overline{x}_i : i\in I]\).
\end{corollary}

\begin{proof}
    By construction, \(\MURI\) is strongly even and satisfies the finiteness condition in \cref{prop:lift}.
    Moreover, \(\MURI\) admits an \(\E^{C_2}_\infty\)-\(\MUR\)-algebra structure by \cref{lemma:MURI}.
    The result follows from \cref{prop:lift}.
\end{proof}

\subsection{Multiplication on quotients of Real bordism}\label{sec:MUR_quotients}

In this section, we prove \cref{main:MUR_quotients}.
The main goal of this section is to endow quotients with a multiplicative structure. Let us begin by recalling the notion of \emph{quotients} in (equivariant) higher algebra.

\begin{definition} \label{definition: quotient}
    Let $A$ be an $\E_{\infty}^G$-algebra and $I$ be an indexing set. Consider classes $x_i \in \pi_{V_i}^G A$ for $G$-representations $V_i$. Then, we define the \tb{quotient} of $A$ by $(x_i)_{i \in I}$ by
    \[ \tb{A/(x_i)_{i \in I}} \coloneqq A \otimes_{A[x_i : i \in I]} A, \]
    with the following relevant maps:
    \begin{enumerate}[(i)]
        \item The augmentation map $A[x_i: i \in I] \to A$ is the $\E_1$-algebra map determined by $x_i \mapsto 0$.\footnote{This is done by writing $A[x_i : i \in I] \simeq \bigotimes_{i \in I}^{\LMod_A} A[x_i]$. Since $A[x_i]$ is the free $\E_1$-algebra, a map $A[x_i] \to A$ is determined by where $x_i$ is sent to. Tensoring these together and multiplying yields the desired map.}  This endows $A$ with a left $A[x_i : i \in I]$-module structure.
        \item The classes $(x_i)_{i \in I}$ determine an $\E_1$-$A$-algebra map $A[x_i : i \in I] \to A$ by sending $x_i$ to $x_i$. In particular, the base change functor $A \otimes_{A[x_i: i \in I]} -$ is defined.
    \end{enumerate}
    This definition agrees with taking a filtered colimit along iterated cofibers, as demonstrated in \cite[Section 2.4.3]{HHR16}. The augmentation map $A[x_i : i \in I] \to A$ in particular induces the projection map $A \to A/(x_i)_{i \in I}$.
\end{definition}

Many important $C_2$-spectra are quotients of $\MUR$, hence the importance of this definition. This relative tensor product description is better suited to discuss multiplicative structure, which is the reason we chose this definition. To study equivariant multiplicative structures with this construction, we will naturally need a robust theory of parametrized left module categories. We developed such a theory in \cite{quinnZhu2026multiplicativeequivariantthomspectra}.

We single out a class of examples that we are particularly interested in, namely the Real Brown--Peterson spectra.

\begin{example} \label{example: BPRJ}
    Fix a choice of indecomposable polynomial generators $\pi_* \MU_{(2)} \cong \Z_{(2)}[x_1, x_2, x_3, \cdots]$ with $|x_i| = 2i$ and write $v_i = x_{2^i-1}$. Since $\MU_{\R(2)}$ is a strongly even ring spectrum, we obtain a lift
    \[ \pi_{*\rho}^{C_2} \MU_{\R(2)} \cong \Z_{(2)}[\ol{x}_1, \ol{x}_2, \ol{x}_3, \cdots] \]
    with $|\ol{x}_i| = i \rho$ and we write $\ol{v}_i = \ol{x}_{2^i -1}$. In particular, one obtains $\BPR \simeq \MUR/(\ol{x}_i : i \neq 2^j -1 \text{ for some j})$ and we have $\pi_{*\rho}^{C_2} \BPR \cong \Z_{(2)}[\ol{v}_1, \ol{v}_2, \ol{v}_3, \cdots]$.

    Now fix a subset $J \subseteq \Z_{\geq 1}$. Then, we define
    \[ \tb{\BPR\langle J \rangle} \coloneqq \BPR/(\ol{v}_j : j \not \in J) \simeq \MUR/(\ol{x}_i, \ol{v}_j : i \neq 2^k - 1, j \not \in J). \]
    This recovers classical examples. Let $n \geq 1$.
    \begin{enumerate}[(i)]
        \item We obtain a form of the Real truncated Brown--Peterson spectra $\BPRn \simeq \BPR\langle \{1,2,\cdots,n \} \rangle$, in the sense of \cite[Definition 3.15]{hillmeier2017}.\footnote{In fact all such objects can be recovered as quotients of $\MUR$ by suitable choices of polynomial generators. This follows from the same argument as \cite[Lemma 2.0.5]{hahnWilson2022redshift}.}
        \item Let $n \geq 1$. Then, $\BPR \langle i \geq n \rangle$ is a Real version of the classical $P(n)$ spectra.
        \item Fix two positive integers $k \leq m$. Then, $\BPR \langle k,m \rangle \coloneqq \BPR \langle i \geq 1 : k \leq i \leq m \rangle$. For example, $\BPR \langle n, n \rangle$ is the integral connective $n$-th Real Morava $K$-theory. 
    \end{enumerate} These are important examples of quotients, which were e.g.~studied in \cite{beaudryhilllawsonshizeng2025slicequotientsnorms}, particularly in the higher group case. We will recall those in \cref{example: BPGJ}.
\end{example}

\begin{lemma} \label{lemma: enhance structure of quotients}
    Let $A$ be an $\E_{\infty}^G$-algebra and $I$ be an indexing set. Consider classes $x_i \in \pi_{V_i}^G A$ for $G$-representations $V_i$. Let $V$ be another $G$-representation. Suppose that $A[x_i : i \in I]$ refines to an $\E_{\R \oplus V}$-algebra and that the structure maps $A[x_i : i \in I] \to A$ from \cref{definition: quotient} refine to $\E_{\R \oplus V}$-algebra maps. Then, $A/(x_i)_{i \in I}$ refines to an $\E_V$-algebra.
\end{lemma}

\begin{proof}
    By restricting along the Dunn map $\E_1 \otimes \E_V \to \E_{\R \oplus V}$, we obtain maps of $\E_1 \otimes \E_V$-algebras \cite[Theorem III.2.2]{StewartDunn}.
    Now, by \cite[Corollary 3.3.9]{quinnZhu2026multiplicativeequivariantthomspectra} the base change functors are $\E_V$-monoidal. So first, consider the composition
    \begin{center}
        \begin{tikzcd}
            \Alg_{\E_V} \left(\ul{\LMod}_{A[x_i:i \in I]} \right) \arrow[rr, "A \otimes_{A[x_i:i \in I]} -"] & & \Alg_{\E_V}(\ul{\LMod}_A) \arrow[rr, "\Res^A_{A[x_i: i \in I]}"] & & \Alg_{\E_V} \left(\ul{\LMod}_{A[x_i:i \in I]} \right)
        \end{tikzcd}
    \end{center}
    Plugging in $A[x_i:i \in I]$ endows $A$ with an $\E_V$-$A[x_i:i \in I]$-algebra structure. Since the base change functor is also $\E_V$-monoidal, the resulting quotient is then an $\E_V$-algebra.
\end{proof}

The main candidate for $A$ is one of the designer polynomial algebras from the previous section, which will be the setting for the following result.

\begin{theorem}\label{thm:MUR_quotients}
    Let \(I\subseteq \Z_{\geq 1}\) and let \(\{\overline{x}_{i}\}_{i\in I}\) be a collection of classes with  \(\overline{x}_i\in \pi^{C_2}_{i\rho}\MUR\).
    Then \(\MUR/(\overline{x}_i:i\in I)\) admits an \(\E_\sigma\)-\(\MUR\)-algebra structure.
\end{theorem}

\begin{proof}
Let \(R \coloneqq \MUR[y_j\vert j\in \Z_{\geq 1}\setminus I]\) with \(|y_j| = j\rho\). 
Then \(R\) is a strongly even \(\E^{C_2}_\infty\)-\(\MUR\)-algebra by \cref{lemma:MURI}.
Moreover, 
\[ R/(\ol{x}_i,y_j: i\in I, j\in\Z_{\geq 1}\setminus I) \simeq \MUR/(\ol{x}_i :i\in I). \]
So, it suffices to construct \(R/(\ol{x}_i,y_j: i\in I, j\in\Z_{\geq 1}\setminus I)\) as an \(\E_\sigma\)-\(R\)-algebra. Then we can forget down to an \(\E_\sigma\)-\(\MUR\)-algebra structure.

In constructing this \(\E_\sigma\)-\(R\)-algebra structure, the difference between \(\ol{x}_i\) and \(y_j\) is unimportant. So we write \(R/(\overline{z}_1,\overline{z}_2,\overline{z}_3,\ldots)\) instead of \(R/(x_i,y_j: i\in I, j\in\Z_{\geq 1}\setminus I)\). By \cref{definition: quotient} this is described as a relative tensor product
\[ R/(\ol{z}_1, \ol{z}_2, \ol{z}_3, \cdots) = R \otimes_{R[\ol{z}_1, \ol{z}_2, \ol{z}_3, \cdots]} R. \]
By \cref{lemma: enhance structure of quotients} we wish to enhance the relevant maps $R[\ol{z}_1, \ol{z}_2, \ol{z}_3, \cdots] \to R$ to $\E_{\rho}$-algebra maps.

Since \(R\) is Real orientable, \(R[\BUR]\) is equivalent to \(R[z_1,z_2,z_3,\ldots]\) with $|z_{i}| = i\rho$ as $\E_1$-$R$-algebras. 
This gives two \(\E_1\)-\(R\)-algebra maps \(R[\BUR]\simeq R[z_1,z_2,z_3,\ldots] \to R\), such that \(R\otimes_{R[\BUR]}R\) is equivalent to \(R/(z_1,z_2,z_3,\ldots)\).
It remains to refine these maps to \(\E_\rho\)-\(R\)-algebra maps, which we do using \cref{prop: lift} now.

By adjunction, an \(\E_1\)-\(R\)-algebra map \(R[\BUR]\to R\) is equivalent to an \(\E_1\)-algebra map \(\Sigma^\infty_+\BUR \to R\). This forgets to a non-equivariant \(\E_1\)-algebra map \(\Sigma^\infty_+\BU\to R^e\). By  \cite[Theorem 7.1]{HahnYuan} this lifts to an \(\E_2\)-map \(\Sigma^\infty_+\BU\to R^e\). By \cref{prop: lift}, it further refines to an \(\E_\rho\)-algebra map \(\Sigma^\infty_+\BUR \to R\).
By adjunction, this defines an \(\E_\rho\)-\(R\)-algebra map \(R[\BUR] \to R\), thus refining the map we started with.
\end{proof}

\begin{remark}
    This proof recovers a version of the classical non-equivariant result that quotients of $\MU$ obtain an $\E_1$-algebra structure, see \cite{hahn2018quotientsrings}.
\end{remark}

In particular, we recover a $C_2$-equivariant version of Basu--Sagave--Schlichtkrull's result \cite[Theorem 5.6]{Basu_2017}, which writes quotients of $\MU$ as Thom spectra over $\SU(n)$ resp.~$\SU$.

\begin{corollary} \label{corollary: Real BSS}
    Let \(I\subseteq \mathbb{Z}_{\geq 1}\) and let \(\{\ol{x}_i\}_{i\in I}\) be a collection of elements with \(\ol{x}_i \in \pi^{C_2}_{i\rho}\MUR\). Then, there exists an $\E_{\sigma}$-algebra map $\mathrm{B}^2 \mathrm{U}_{\R} \to \BGL_1 R$ such that
    \[ \MU_{\R}/(\ol{x}_i : i \in I) \simeq \Th_{C_2}(\mathrm{B}^2 \mathrm{U}_{\R} \to \BGL_1 R). \]
    Here, $R = \MUR[y_j\vert j\in \Z_{\geq 1}\setminus I]$ with \(|y_j| = j\rho\) is the algebra from \cref{thm:MUR_quotients}.
\end{corollary}

\begin{proof}
    In \cref{thm:MUR_quotients} we showed 
    \[ \MU_{\R}/(x_i : i \in I) \simeq R\otimes_{R[\BU_{\R}]} R \]
    and in particular obtained an $\E_{\rho}$-$R$-algebra map $R[\BU_{\R}] \to R$ to define this base change. By adjunction, it corresponds to an $\E_{\sigma}$-algebra map $\B^2 U_{\R} \to \BGL_1 R$. Since $\Th_{C_2}$ is symmetric monoidal and colimit-preserving, we can massage the Thom spectrum of a bar construction into the desired relative tensor product: $\Th_{C_2}(\mathrm{B}^2 \mathrm{U}_{\R} \to \BGL_1 R) \simeq R\otimes_{R[\BU_{\R}]} R$.
\end{proof}

\begin{remark} \label{remark: BSS hard}
    Note that $\mathrm{B^2} \mathrm{U}_{\R} \simeq \mathrm{B} \Omega^{\sigma} \SU_{\R}$. This restricts to $\SU$ on underlying, and explains the difficulty of generalizing \cite[Theorem 5.6]{Basu_2017} to the Real setting. In doing so, one needs to take $\mathrm{B}\Omega^{\sigma} \SU_{\R}$ as a $C_2$-equivariant lift of $\SU$, not $\SU_{\R}$. However, even with this realization, it seems hard to naively mimic Basu--Sagave--Schlichtkrull's arguments to the $C_2$-equivariant setting without lifting techniques such as the ones we employ here.
    
    The main result \cite[Theorem 5.6]{Basu_2017} builds on the classical map $\mathbb{CP}^{n-1} \to \Omega \SU(n)$, which is e.g.~defined in \cite[Introduction]{HahnYuan}. We checked that this enhances to a $C_2$-map $\CPR^{n-1} \to \Omega^{\sigma} \SU_{\R}(n)$. However, several complications still show up in trying to naively mimic the strategy in \cite{Basu_2017}.

    Nonetheless, we expect that this $C_2$-enhancement along with \cref{prop: lift} is a first step towards a Real equivariant version of Hahn--Yuan's result, and thereby gives a direction towards \cite[Question 4]{HahnYuan}.
\end{remark}

\begin{remark} \hfill \label{remark: MUR quotient theorem}
    \begin{enumerate}[(i)]
        \item More generally, we can base change along any $\E_{\rho}$-map $\MUR \to E$. By \cite[Corollary 6.2.2]{quinnZhu2026multiplicativeequivariantthomspectra} we are thus able to give an $\E_{\sigma}$-structure on quotients of strongly even $\E_{\infty}^{C_2}$-rings by homogeneous elements coming from $\MUR$.
        \item So far we have worked integrally. All constructions and results used have obvious \((2)\)-local analogues. 
        Moreover, the proofs in the \((2)\)-local setting follow mutatis mutandis from the integral ones. This gives a \((2)\)-local version of \cref{thm:MUR_quotients}.
        \item Let $i > 0$, then $\BPR/v_i$ does not admit the structure of a homotopy associative ring \cite[Example 3.15]{bachmannHahn2022nilpotencenormedmglmodules}, so in particular there is no $\E_1$-analog of \cref{thm:MUR_quotients}.
    \end{enumerate}
\end{remark}

As an immediate application, we obtain an $\E_{\sigma}$-$\MUR$-algebra structure on $C_2$-spectra of interest.
\begin{corollary}\label{cor:BPRnkRn}
        Let $J \subseteq \Z_{\geq 1}$. Then, $\BPR \langle J \rangle$ admits an $\E_{\sigma}$-$\MUR$-algebra structure. 
\end{corollary}

By \cref{example: BPRJ} this in particular includes the truncated Real Brown--Peterson spectra $\BPRn$ and the integral connective Real Morava $K$-theories.

One of the main interests in \(\E_\sigma\)-algebra structures is their connection to Real trace methods.
An \(\E_\sigma\)-algebra structure allows one to define Real topological Hochschild homology. As an immediate corollary of \cref{thm:MUR_quotients}, we will be able to apply \(\THR(-)\) to such quotients of \(\MUR\).

\begin{remark}
    Both \(\THR(\BPRn)\) and \(\THR(\BPRn/\MUR)\) are defined for all \(n\).
\end{remark}
This should be the starting point of running the story of \cite{gabe2025realsyntomiccohomology} for $\BPRn$ with $n \geq 3$.

\subsection{Multiplication on twisted monoid quotients of normed bordism}\label{sec:MUG_quotients}
Let us recall Hill--Hopkins--Ravenel's method of \emph{twisted monoid rings} from \cite[Section 2.4.2]{HHR16}, which yields a normed variant of the classical quotient notion (\cref{definition: quotient}).
We will work in less generality.
We specialize their construction to the case where the spectrum we want to form twisted monoid ring quotients of is the Hill--Hopkins--Ravenel norm of a \(\E^{C_2}_\infty\)-ring spectrum.

Due to the similarity of the definitions, the arguments will also be almost identical. Nonetheless, we will need to spell these out again, since twisted monoid ring quotients are not a special case of quotients.
\begin{construction}[{\cite[Section 2.4]{HHR16}}]\label{construction:TMR2}
    Fix a finite group \(G \geq C_2 \).
    Let \(R\) be an \(\E^{C_2}_\infty\)-ring spectrum. Moreover, let \(I\subseteq \Z_{\geq 1}\) and let \(\{\ol{x}_i\}_{i\in I}\) be a collection of elements with \(\overline{x}_i\in \pi^{C_2}_{i\rho} \RG\). 
    \begin{enumerate}[(i)]
        \item We write \(\tb{\RG} \coloneqq N^G_{C_2}R\) and $\tb{R^{( \! ( G ) \! )}[G \cdot \ol{x}_i : i \in I]} \coloneqq N_{C_2}^G(R[\ol{x}_i : i \in I])$. 
        \item We define the \tb{twisted monoid ring quotient} as the relative tensor product 
        \[
            \tb{\RG/(G\cdot \overline{x}_i:i\in I)} \coloneqq \RG\otimes_{\RG[G\cdot \overline{x}_i:i\in I]}\RG,
        \]
        consisting of the following structure:

        First, sending $\ol{x}_i$ to $0$ determines augmentation maps $R[\ol{x}_i : i \in I] \to \Res_{C_2}^G R^{( \! ( G ) \! )}$ as $\E_1$-$R$-algebra maps. The composite
        \[
        \RG[G\cdot \overline{x}_i:i\in I] = N^G_{C_2}(R[\overline{x}_i:i\in I])\longrightarrow N^{G}_{C_2}\mathrm{Res}^G_{C_2}\RG\longrightarrow \RG
         \]
        is an $\E_1$-algebra map. This gives $R^{( \! ( G ) \! )}$ the structure of a left $\RG[G\cdot \overline{x}_i:i\in I]$-module.

        On the other hand, the elements $\ol{x}_i \in \pi_{i\rho}^{C_2} R^{( \! ( G ) \! )}$ define a $\Coind_{C_2}^G \E_1$-$R^{( \! ( G ) \! )}$-algebra map 
        \[ \RG[G\cdot \overline{x}_i:i\in I] = N^G_{C_2}(R[\overline{x}_i:i\in I])\longrightarrow N^{G}_{C_2}\mathrm{Res}^G_{C_2}\RG\longrightarrow \RG \] in the same fashion, so we can base change along this.
    \end{enumerate}
    The augmentation map $\RG[G\cdot \overline{x}_i:i\in I] \to \RG$, i.e.~the first of these composites, induces the projection map $\RG \to \RG/(G \cdot \ol{x}_i : i \in I)$.
\end{construction}

In principle, we could have phrased \cref{sec:MUR_quotients} in this generality and recovered the results there as a special case for $G = C_2$. For readability we decided to split it into two parts. 

In \cref{section: twisted monoid quotients via parametrized colimits} we give an alternative viewpoint of twisted monoid quotients through parametrized colimits. Working with that definition could have been a different entry point to the theory, but we have decided to keep close to the classical language.

This subsection concerns multiplicative structures on normed equivariant ring spectra, which are naturally given by coinduced equivariant operads, see e.g.~\cite[Construction 2.2.5]{quinnZhu2026multiplicativeequivariantthomspectra}. We will essentially only need \cite[Construction 2.2.5, Corollary 6.1.6]{quinnZhu2026multiplicativeequivariantthomspectra} about this construction.

\begin{lemma} \label{lemma: enhance structure of twisted monoid quotients}
    Let $A$ be an $\E_{\infty}^{C_2}$-algebra and $I$ be an indexing set. Consider classes $\ol{x}_i \in \pi_{V_i}^{C_2} A^{( \! ( G ) \! )}$ for $C_2$-representations $V_i$. Let $V$ be another $C_2$-representation. Suppose that $A[\ol{x}_i : i \in I]$ refines to an $\E_{\R \oplus V}$-algebra and that the structure maps 
    \[ A[\ol{x}_i : i \in I] \longrightarrow \Res_{C_2}^G A^{( \! ( G ) \! )} \] from \cref{construction:TMR2} refine to $\E_{\R \oplus V}$-algebra maps.
    Then, $A^{( \! ( G ) \! )}/(G \cdot x_i)_{i \in I}$ refines to a $\Coind_{C_2}^G\E_V$-algebra.
\end{lemma}

\begin{proof}
    Restricting $\E_{\R \oplus V}$ along the Dunn map yields $\E_1 \otimes \E_V$-structures \cite[Theorem III.2.2]{StewartDunn}. By \cite[Construction 2.2.5]{quinnZhu2026multiplicativeequivariantthomspectra} the $G$-spectrum $A^{( \! ( G ) \! )}[G \cdot \ol{x}_i]$ obtains a $\Coind_{C_2}^G(\E_1 \otimes \E_V)$-algebra structure and the composite maps
    \[
        A^{( \! ( G ) \! )}[G\cdot \overline{x}_i:i\in I] = N^G_{C_2}(A[\overline{x}_i:i\in I])\longrightarrow N^{G}_{C_2}\mathrm{Res}^G_{C_2}A^{( \! ( G ) \! )} \longrightarrow A^{( \! ( G ) \! )}
    \]
    refine to $\Coind_{C_2}^G(\E_1 \otimes \E_V)$-algebra maps. We now check that there is a map
    \[ \E_1 \otimes \Coind_{C_2}^G \E_V \longrightarrow \Coind_{C_2}^G(\E_1 \otimes \E_V) \]
    of $G$-$\infty$-operads. Afterwards, we obtain the desired algebra structures by base changing as in \cref{lemma: enhance structure of quotients}.

    By adjunction, such a map is equivalent to a map
    \[ \Res_{C_2}^G (\E_1 \otimes \Coind_{C_2}^G \E_{V}) \longrightarrow \E_1 \otimes \E_{V}. \]
    Stewart proved that $\ul{\Op}_{G, \infty}$ has a $G$-symmetric monoidal enhancement \cite[Corollary E']{stewart2025tensorproductsequivariantcommutative}. So the associated restriction maps are symmetric monoidal \cite[Lemma A.1.1]{quinnZhu2026multiplicativeequivariantthomspectra}. Thus, we may take the map
    \[ \Res_{C_2}^G (\E_1 \otimes \Coind_{C_2}^G \E_{V}) \simeq \E_1 \otimes \Res_{C_2} \Coind_{C_2}^G \E_{V} \xrightarrow{\E_1 \otimes \varepsilon_{\E_{V}}} \E_1 \otimes \E_{V} \]
    induced by the counit.
\end{proof}

\begin{remark}
    One can set up twisted monoid quotients in more generality for arbitrary $\E_{\infty}^G$-ring spectra and subgroups $H \leq G$ following Hill--Hopkins--Ravenel \cite{HHR16}. We chose to restrict to examples $R^{( \! ( G ) \! )}$ to be able to phrase \cref{lemma: enhance structure of twisted monoid quotients} in its current form.
\end{remark}

\begin{theorem}\label{thm:MUG_quotients}
    Let \(I\subseteq \Z_{\geq 1}\) and let \(\{\overline{x}_{i}\}_{i\in I}\) be a collection of elements with \(\overline{x}_i\in \pi^{C_2}_{i\rho}\MUG\).
    Then, \(\MUG/(G\cdot \overline{x}_i:i\in I)\) admits a \(\Coind^G_{C_2}\E_\sigma\)-\(\MUG\)-algebra structure.
\end{theorem}
\begin{proof}
    Let \(R \coloneqq \MUR[\overline{y}_j:j\in \Z_{\geq 1}\setminus I]\). 
    It is an \(\E^{C_2}_\infty\)-\(\MUR\)-algebra by \cref{lemma:MURI}.
    Then, 
    \[ \RG \simeq \MUG[G\cdot\overline{y}_j:j\in \Z_{\geq 1}\setminus I]. \]
    Hence, 
    \begin{align*}
        \MUG/(G\cdot\overline{x}_i :i\in I) &\simeq \RG/(G\cdot\overline{x}_i, G\cdot \overline{y}_j :i\in I, j\in \Z_{\geq 1}\setminus I)
        \\ &= \RG \otimes_{\RG[G \cdot \ol{x}_i, G \cdot \ol{y}_j : i \in I, j \in \Z_{\geq 1} \setminus I]} \RG.
    \end{align*}
    The difference between \(\overline{x}_i\) and \(\overline{y}_j\) is not important for the purpose of constructing multiplicative structures. 
    So instead of \(\RG/(G\cdot\overline{x}_i, G\cdot \overline{y}_j :i\in I, j\in \Z_{\geq 1}\setminus I)\) we write \(\RG/(G\cdot\overline{z}_1,G\cdot\overline{z}_2,\ldots)\).

    By \cref{lemma: enhance structure of twisted monoid quotients} we need to enhance the relevant map \(R[\overline{z}_1,\overline{z}_2,\ldots ]\to \mathrm{Res}^G_{C_2}\RG\) to an \(\E_\rho\)-\(R\)-algebra map.
    As in the proof of \cref{thm:MUR_quotients}, the source is equivalent to $R[\BUR]$, so this follows from \cref{prop:lift}.
\end{proof}

\begin{remark}
    As in the previous subsection, \cref{thm:MUG_quotients} has an obvious \((2)\)-local version.
\end{remark}

The main interest in the twisted monoid ring construction lies in the case $G = C_{2^n}$, where it has been used to construct the higher truncated Brown--Peterson spectra. To discuss its multiplicative structure, we will introduce another base change.

\begin{definition}[{\cite[Section 2.4.3]{HHR16}}] \label{definition: module twisted monoid quotient}
    Let \(R\) be an \(\E^G_\infty\)-ring and let $R \to A$ be an $\E_1$-algebra map. Let $I$ be an indexing set and let \(\{\overline{x}_{i}\}_{i\in I}\) be a collection of elements with \(\overline{x}_i\in \pi^{C_2}_{i\rho}R\).  Then, we define 
    \[ \tb{A/(G\cdot \overline{x}_{i} : i\in I)}\coloneqq A \otimes_R R/(G\cdot \overline{x}_{i} : i\in I). \]
\end{definition}

\begin{example}[{\cite[Section 2]{beaudryhilllawsonshizeng2025slicequotientsnorms}}] \label{example: BPGJ}
    Let $G = C_{2^n}$. By \cite[Section 5]{HHR16} we can fix a choice of generators such that 
    \[ \pi^{C_2}_{*\rho}\BPG\cong\Z_{(2)}[G\cdot \overline{t}_1,G\cdot \overline{t}_2,\ldots]. \]
    Since \(\BPG\) is a summand of \(\MUGtwo\), we may view these elements as elements in \(\pi^{C_2}_{*\rho}\MUGtwo\). Now let $J \subseteq \Z_{\geq 1}$. Then, 
    \[ \tb{\BPG\langle J \rangle} \coloneqq \BPG/(G \cdot \ol{t}_j : j \not \in J) \simeq \BPG \otimes_{\MU^{( \! ( G ) \! )}_{(2)}} \MU^{( \! ( G ) \! )}_{(2)}/(G \cdot \ol{t}_j : j \not \in J) \]
    In particular, this recovers \tb{(forms of) higher truncated Brown--Peterson spectra} as
    \[ \tb{\BPGm} \coloneqq \BPG\langle \{1, \cdots, m \} \rangle \simeq \BPG/(G \cdot \ol{t}_{m+1}, G \cdot \ol{t}_{m+2}, \cdots). \]
\end{example}

\begin{lemma}\label{lem:basechange}
    Let \(R\) be a \(\E^{G}_\infty\)-ring and let \(\O^{\otimes} \) be a $G$-$\infty$-operad.
    Suppose that $M$ is an $R$-module through an \(\E_1 \otimes \O\)-algebra map \(R\to M\).
    If \(R/(G\cdot \overline{x}_{i} : i\in I)\) admits an \(\O\)-\(R\)-algebra structure,
    then \(M/(G\cdot \overline{x}_{i} : i\in I)\) admits an \(\O\)-\(M\)-algebra structure.
\end{lemma}

\begin{proof}
    The base change along $R \to M$ is $\O$-monoidal by \cite[Corollary 3.3.9]{quinnZhu2026multiplicativeequivariantthomspectra}.
\end{proof}

Our techniques now yield:

\begin{corollary}
    Let \(G=C_{2^n}\) and $J \subseteq \Z_{\geq 1}$. Then, $\BPG\langle J \rangle$ admits a $\Coind_{C_2}^G \E_{\sigma}$-$\BPG$-algebra structure.
\end{corollary}

This in particular provides a first multiplicative structure on the higher truncated Brown--Peterson spectra $\BPGm$.

\section{Orientations by twisted monoid quotients} 
\label{section: orientations by twisted monoid quotients}
\subsection{Factorization through twisted monoid quotients}
Having just studied the structure on twisted monoid quotients, we are motivated to study maps out of these objects. Our main objective is to orient Lubin--Tate theory by higher truncated Brown--Peterson spectra. Let us begin in a general setting, although we will quickly specialize to the examples of interest.

\begin{proposition} \label{prop: factorization}
    Let \(A\) and \(B\) be \(\E^G_\infty\)-ring spectra and \(f\colon A\to B\) be a \( \Coind_H^G \mathbb{E}_{1}\)-algebra map. 
    Fix \(H\leq G\), and let \(\{x_i\}_{i\in I}\) be a collection of elements in \(\pi^H_\star(A)\), i.e.,~maps \(x_i\colon \mathbb{S}^{V_i}\to \mathrm{Res}^G_H A\).
    Suppose that \(f(x_i)=0\) in \(\pi^H_\star(B)\) for all \(i\in I\).
    Then, there is a factorization
    \begin{center}
        \begin{tikzcd}
        A \ar[r, "f"]\ar[d] & B \\
        A/(G\cdot x_i : i\in I) \ar[ru, dashed]
    \end{tikzcd}
    \end{center}
    of $A$-modules, where $A \to A/(G\cdot x_i : i\in I)$ is the projection map and $B$ is an $A$-module through $f$.
\end{proposition}

\begin{proof}
    Composing the augmentation with $f$ yields a map
    \begin{center}
        \begin{tikzcd}
            {A[G \cdot x_i : i \in I]} \arrow[r] & A \arrow[r, "f"] & B
        \end{tikzcd}
    \end{center}
    of $\E_1$-algebras. Functoriality of the bar construction yields a commutative diagram
    \begin{center}
    \begin{tikzcd}
        A \ar[r, "f"]\ar[d] & B\ar[d] \\
        A\otimes_{A [G\cdot x_i:i \in I]} A\ \ar[r, "f", swap] & B\otimes_{A [G\cdot x_i:i \in I]} A.
    \end{tikzcd}
    \end{center}
    By definition (\cref{construction:TMR2}, \cref{definition: module twisted monoid quotient}) we have 
    \[ A/(G\cdot x_i : i\in I) = A\otimes_{A [G\cdot x_i:i \in I]} A \quad \text{and} \quad B/(G\cdot f(x_i):i\in I) = B\otimes_{A [G\cdot x_i:i \in I]} A \simeq B \otimes_{\S[G \cdot x_i: i \in I]} \S. \]
    To finish the proof, it remains to show that there is a map \(B/(G\cdot f(x_i):i\in I)\to B\) that is a retract of the projection map $B \to B/(G\cdot f(x_i):i\in I)$.

    By assumption, \(f(x_i)=0\) in \(\pi^H_\star(B)\) for all \(i\in I\). First, we consider the diagram
    \begin{center}
        \begin{tikzcd}
            {\S[G \cdot x_i: i \in I]} \arrow[r] & N_H^G \Res_H^G A \arrow[r, "f"] \arrow[d] & N_H^G \Res_H^G B \arrow[d]
            \\ & A \arrow[r, "f", swap] & B
        \end{tikzcd}
    \end{center}
    where the top left map is determined by $x_i \in \pi_{\star}^H(A)$. It commutes since $f$ is a $\Coind_H^G \E_{1}$-algebra map.\footnote{\label{footnote: Coind}Using \cite[Corollary 1.39]{stewart2025tensorproductsequivariantcommutative} one can compute the structure spaces of a coinduced operad in terms of a limit along structure spaces of the original operad. In particular, the structure space $(\Coind_H^G \E_1)(G/H)$ is the one giving rise to an $H \to G$ norm multiplication. See \cite[Introduction]{stewart2025equivariantoperadssymmetricsequences} for more information. The limit is indexed over a diagram with an initial object, so one manages to compute $(\Coind_{H}^G \E_1)(G/H) \simeq \E_1(\Res_H^G G/H)$. This is non-empty. Thus, the $\Coind_H^G \E_1$-algebra structure yields compatibility with norm multiplications.} So
    \[ B/(G\cdot f(x_i) : i\in I) \simeq B \otimes_{\S[G\cdot x_i : i\in I]}\S \simeq B \otimes_{N^G_H \mathrm{Res}^G_H B} {N^G_H \mathrm{Res}^G_H B}\otimes_{\S[G\cdot x_i : i\in I]}\S. \]
    Using that $N_H^G$ is symmetric monoidal and sifted colimit-preserving, we furthermore obtain
    \begin{align*}
        B/(G\cdot f(x_i) : i\in I) &\simeq B \otimes_{N^G_H \mathrm{Res}^G_H B} {N^G_H \mathrm{Res}^G_H B}\otimes_{\S[G\cdot x_i : i\in I]}\S 
        \\ &\simeq B \otimes_{N^G_H \mathrm{Res}^G_H B} N^G_H\left(\mathrm{Res}^G_H B/( f(x_i) : i\in I)\right).
    \end{align*}
    We now use the equivalent description of quotients through iterated cofibers, see \cref{definition: quotient}. Since $f(x_i) = 0$, these cofibers split as left $\Res_H^G B$-modules, and we find
    \[ \Res_H^G B/(f(x_i) : i \in I) \simeq \Res_H^G B \oplus \bigoplus_{W} \Sigma^{W} \Res_H^G B. \]
    as $\Res_H^G B$-modules for certain $H$-representations $W$. Projecting to the first summand, we have constructed a map
    \[ B/(G\cdot f(x_i) : i\in I) \longrightarrow B \otimes_{N_H^G \Res_H^G B} N_H^G \Res_H^G B \simeq B. \]
    of $B$-modules. By construction, this is a retract of the projection $B \to B/(G \cdot f(x_i) : i \in I)$.
\end{proof}

\begin{corollary}
    \label{corollary: factoring BPG to E}
    Let $E$ be an $\E_{\infty}^G$-ring spectrum and $f \colon \BPG \to E$ be a $\Coind_{C_2}^G \E_{1}$-algebra map. Let $J \subseteq \Z_{\geq 1}$ and suppose that $\ol{t}_j \in \pi_{*\rho}^{C_2} \BPG$ is sent to $0$ along $f$ for $j \not \in J$. Then, there is a factorization
    \begin{center}
        \begin{tikzcd}
            \BPG \arrow[r, "f"] \arrow[d] & E
            \\ \BPG \langle J \rangle \arrow[ur, dashed, bend right]
        \end{tikzcd}
    \end{center}
    of $\BPG$-modules, where $E$ is a $\BPG$-module through $f$.
\end{corollary}

\begin{proof}
    Fix a $\Coind_{C_2}^G \E_1$-retract $r \colon \MUG \to \BPG$; such maps exist by \cite[Theorem 6.3.1]{quinnZhu2026multiplicativeequivariantthomspectra}. Consider the commutative diagram
    \begin{center}
        \begin{tikzcd}
        \Map_{\LMod_{\MUG}}\left(\MUG/(G \cdot t_j : j \not \in J), r^* E \right) \arrow[r] \arrow[d, "\simeq", no head, swap] & \Map_{\LMod_{\MUG}}\left(\MUG, r^* E \right) \arrow[d, "\simeq", no head]
        \\ \Map_{\LMod_{\BPG}} \left(\BPG\langle J \rangle, E \right) \arrow[r] & \Map_{\LMod_{\BPG}}\left(\BPG, E \right)
        \end{tikzcd}
    \end{center}
    induced by the base change/restriction adjunction. The top arrow is surjective on $\pi_0$ by the previous proposition (\cref{prop: factorization}). Thus, so is the bottom arrow. It yields precisely a factorization as indicated.
\end{proof}

One of the main interests in the twisted monoid construction is the relation to Lubin--Tate theories and higher real \(K\)-theories.

Building on work of Hahn--Shi \cite{hahnRealOrientationsLubin2020}, 
Beaudry--Hill--Shi--Zeng construct \(G\)-equivariant homotopy ring maps from \(\BPG\) to certain explicit forms of Lubin--Tate theories \(E(k,\Gamma_h)\), see \cite{beaudryHillShiZeng2021modelsLubinTate}.
In this section, we refine this to produce \(G\)-equivariant maps \(\BPGm\to E(k,\Gamma_h)\). Since our work relies heavily on that of Beaudry--Hill--Shi--Zeng, first we will recall some essential parts of their work.

\begin{recollection}[{\cite{beaudryHillShiZeng2021modelsLubinTate}}] \label{recollection: bhsz}
    Let \(G=C_{2^n}\) and fix a height \(h=2^{n-1}m\).
    Let \(E(k,\Gamma_h)\) be any of the forms of Lubin--Tate theory constructed in \cite[Theorem 1.5]{beaudryHillShiZeng2021modelsLubinTate}.
    Beaudry--Hill--Shi--Zeng construct \(G\)-equivariant maps
    \[\BPG\longrightarrow E(k,\Gamma_h)\longrightarrow E(k,\Gamma_h)^{hC(k,m)} \quad \text{with} \quad C(k,m) \coloneqq \operatorname{Gal}(k/\F_2) \ltimes k^\times [q] \subseteq \mathbb{G}(k, \Gamma_h),
    \]
    where $k^{\times}[q] \subseteq k^{\times}$ is the subgroup of $q = (2^m-1)$-torsion elements, see \cite[Remark 6.1]{beaudryHillShiZeng2021modelsLubinTate}.
    They construct a class \(D\in \pi^{G}_{*\rho_G}(\MUG)\),\footnote{The precise definition won't be relevant to us, the desiderata for it are explained in \cite[Section 6]{beaudryHillShiZeng2021modelsLubinTate}. Nonetheless, let us briefly recall the construction. Consider those elements $x \in \pi_{*\rho}^{C_2} \BPG$ that become invertible under $\pi_{*\rho}^{C_2}\BPG \to \pi_{*\rho}^{C_2} E_h$. These can be viewed as elements in $\pi_{*\rho}^{C_2} \MUG$ under the inclusion $\BPG \to \MUG$ induced by the Real Quillen idempotent. Then, we define $D \coloneqq \prod_x N_{C_2}^G x \in \pi_{*\rho_G}^G \MUG$.} and show that there is a
    \(G\)-equivariant factorization
    \[
    \begin{tikzcd}
        \BPG \ar[r]\ar[d] & E(k,\Gamma_h) \ar[r] & E(k,\Gamma_h)^{hC(k,m)} \\
        D^{-1}\BPG \ar[ru, dashed]
    \end{tikzcd}
    \]
    see \cite[Theorem 1.8]{beaudryHillShiZeng2021modelsLubinTate}, where $E(k,\Gamma_h) \to E(k,\Gamma_h)^{hC(k,m)}$ is a $G$-equivariant map splitting the natural map $E(k, \Gamma_h)^{hC(k,m)} \to E(k,\Gamma_h)$, see \cite[Theorem 5.3]{beaudryHillShiZeng2021modelsLubinTate}.
    Furthermore, they show that there are non-equivariant factorizations
    \[
    \begin{tikzcd}
        \Res^G_{e}\BPG \ar[r]\ar[d] & \Res^G_{e}E(k,\Gamma_h) \ar[r] & \Res^G_{e}E(k,\Gamma_h)^{hC(k,m)} \\
        \Res^G_{e}\BPGm \ar[ru, dashed]
    \end{tikzcd}
    \]
    and similarly
    \[
    \begin{tikzcd}
        \Res^G_{e}D^{-1}\BPG \ar[r]\ar[d] & \Res^G_{e}E(k,\Gamma_h) \ar[r] & \Res^G_{e}E(k,\Gamma_h)^{hC(k,m)} \\
        \Res^G_{e}D^{-1}\BPGm \ar[ru, dashed]
    \end{tikzcd}
    \]
    see \cite[Proposition 7.6]{beaudryHillShiZeng2021modelsLubinTate}. Additionally, they show that there is a non-equivariant \(K(h)\)-local equivalence 
    \[
    L_{K(h)}\Res^G_{e}D^{-1}\BPGm \xrightarrow{\ \simeq \ } \Res^G_{e}E(k,\Gamma_h)^{hC(k,m)}
    \]
    see \cite[Proposition 7.6]{beaudryHillShiZeng2021modelsLubinTate}.
\end{recollection}

Using the orientations produced in \cite{quinnZhu2026multiplicativeequivariantthomspectra}, we upgrade Beaudry--Hill--Shi--Zeng's factorizations to an equivariant factorization. Furthermore, we use this to obtain a model of higher real $K$-theories.

\begin{theorem}\label{thm:BHSZ_refinement}
    The Beaudry--Hill--Shi--Zeng orientation \(\BPG \to E(k,\Gamma_h)\) refines to a \(\Coind^G_{C_2}\E_\rho\)-ring map.
    Furthermore, there are \(G\)-equivariant factorizations
    \[
    \begin{tikzcd}
        \BPG \ar[r]\ar[d] & E(k,\Gamma_h)
        &
        \raisebox{-0.9em}{\text{and}}
        &
        D^{-1}\BPG \ar[r]\ar[d] & E(k,\Gamma_h) \\
        \BPGm \ar[ru, dashed]
        &&
        &
        D^{-1}\BPGm \ar[ru, dashed]
    \end{tikzcd}
    \]
    of $\BPG$-modules.
\end{theorem}
\begin{proof}
    The refinement of the Beaudry--Hill--Shi--Zeng orientation \(\BPG\to E(k,\Gamma_h)\) to a \(\Coind^G_{C_2}\E_\rho\)-ring map follows from the lifting theorem \cite[Corollary 6.1.6]{quinnZhu2026multiplicativeequivariantthomspectra} combined with the structured retraction of $\BPR$ constructed in \cite[Theorem 6.3.3]{quinnZhu2026multiplicativeequivariantthomspectra}. 
    The Beaudry--Hill--Shi--Zeng orientation is constructed so that the elements \(\overline{t}_{m+1},\overline{t}_{m+2},\overline{t}_{m+3},\ldots\in \pi^{C_2}_*(\BPG)\), that are used in \cref{example: BPGJ} to define \(\BPGm\), all map to zero in \(E(k,\Gamma_h)\).

    Therefore, \cref{corollary: factoring BPG to E} yields the factorization on the left. Since $D$ is sent to an invertible element in $E(k,\Gamma_h)$, applying $D^{-1}$ yields the factorization on the right. 
\end{proof}

\begin{remark}
    Under the more restrictive assumption that the Beaudry--Hill--Shi--Zeng orientation can be constructed starting from an \(\E_2\)-ring map \(\MUR \to \Res^{G}_{C_2}E(k,\Gamma_h)\), an alternative approach for producing such factorizations was given in \cite[Proposition 3.2.6]{delangel2024dualshigherrealktheories} suggested to Del Angel by Hahn. Even though we can produce some $\E_{2\rho}$-algebra maps $\MUR \to \Res_{C_2}^G E(k, \Gamma_h)$ as in \cite[Example 2.21]{quinnZhu20206realsnaith}, we emphasize that it is a non-trivial task to lift the specific orientations considered by Beaudry--Hill--Shi--Zeng. In particular, we are not able to produce $\E_{2\rho}$-lifts of these orientations. These difficulties are related to lifting the Quillen idempotent to an $\E_4$-algebra map, which still stands as an open problem.
\end{remark}

\subsection{Models of higher real $K$-theories}

Recall that we are considering certain examples $E(k, \Gamma_h)$ of Lubin--Tate theories considered by Beaudry--Hill--Shi--Zeng. Our equivariant map (\cref{thm:BHSZ_refinement}) allows us to refine Beaudry--Hill--Shi--Zeng's underlying equivalence to a \(G\)-equivariant equivalence. Let us denote by $\ul{\map}_G(-,-)$ the mapping $G$-spectrum. 

\begin{corollary}\label{cor:BPGmEh_equivalence}
    Let $G = C_{2^n}$ and $h = 2^{n-1}m$. Suppose that $k^{\times}$ contains all $q = (2^m-1)$-roots of unity. Then, there is an equivalence
    \[ \ul{\map}_G \left(\Sigma_+^{\infty} \mathrm{EG}, L_{\Infl_e^G K(h)} D^{-1}\BPGm \right) \xrightarrow{\ \simeq \ } E(k,\Gamma_h)^{hC(k,m)} \]
    of $\BPG$-modules.
\end{corollary}
\begin{proof}
    Note that \(E(k,\Gamma_h)^{hC(k,m)}\) is Borel, and the underlying non-equivariant spectrum is \(K(h)\)-local.
    We write $\ul{\map}_G(-,-)$ for the mapping $G$-spectrum. Applying $\ul{\map}_G \left(\Sigma_+^{\infty} \mathrm{EG}, L_{\Infl_e^G K(h)}- \right)$ to the map from \cref{thm:BHSZ_refinement} thus yields a map
    \[
    \ul{\map}_G \left(\Sigma_+^{\infty} \mathrm{EG}, L_{\Infl_e^G K(h)} D^{-1}\BPGm \right) \longrightarrow E(k,\Gamma_h)^{hC(k,m)}.
    \]
    of $G$-spectra. By construction, this is a map between Borel 
    spectra. In particular, it is an equivalence if and only if it is an equivalence on underlying non-equivariant spectra.
    The underlying non-equivariant equivalence is \cite[Proposition 7.6]{beaudryHillShiZeng2021modelsLubinTate}. Here, we use \cite[Proposition 3.2(2)]{carrickSmashingLocalizationsEquivariant2022} to see that $L_{\Infl_e^G K(h)}$ becomes $L_{K(h)}$ on underlying.
\end{proof}

\begin{remark}
    An alternative viewpoint on the functor we applied to $D^{-1} \BPGm$ above is through an equivariant chromatic localization, namely
    \[  L_{G_+\otimes K(h)}D^{-1}\BPGm \simeq \ul{\map}_G \left(\Sigma_+^{\infty} \mathrm{EG}, L_{\Infl_e^G K(h)} D^{-1}\BPGm \right) \]
    according to \cite[Proposition 3.21]{carrickSmashingLocalizationsEquivariant2022}.
\end{remark}

Using the previous result, we can now deduce a model for higher real $K$-theories.

\begin{corollary}\label{cor:fixed_points_EOn_equivalence}
    Let $G = C_{2^n}$ and $h = 2^{n-1}m$. Suppose that $k^{\times}$ contains all $q = (2^m-1)$-roots of unity. For any \(H\leq G\), there is an equivalence
    \[  
    \EO_h(C(k,m) \times H) \coloneqq \left(E(k,\Gamma_h)^{hC(k,m)}\right)^{hH} \simeq
    \left(L_{K(h)}D^{-1}\BPGm\right)^{hH}  \] 
    of spectra.
\end{corollary}
\begin{proof}
    This follows directly from applying \((-)^{hH}\) to the equivalence of \cref{cor:BPGmEh_equivalence}.
\end{proof}

We give the following application:

\begin{corollary}
    Let $G = C_{2^n}$ and $h = 2^{n-1}m$. Suppose that $k^{\times}$ contains all $q = (2^m-1)$-roots of unity. Then, $L_{K(h)}\left(D^{-1}\BPGm^{H}\right)\neq 0$ for all $H \leq G$.
\end{corollary}

\begin{proof}
    Since $D^{-1} \BPGm$ is an \(\MUG\)-module, it follows from \cite[Theorem 2.14]{carrick2025higherrealktheoriesfinite} that we have 
    \[ L_{K(h)}\left(D^{-1}\BPGm^{H}\right) \simeq L_{K(h)}\left(D^{-1}\BPGm^{hH}\right), \]
    which we need to show to be non-zero. The map $\Sigma_+^{\infty} \mathrm{EH} \to \S$ induces a map
    \[ D^{-1} \BPGm \longrightarrow \ul{\map}_H \left(\Sigma_+^{\infty} \mathrm{EH}, D^{-1}\BPGm \right), \]
    where $\ul{\map}_H$ denotes the mapping $H$-spectrum. The naturality of the norm map $(-)_{hH} \Rightarrow (-)^{hH}$ yields a commutative diagram
    \begin{center}
        \begin{tikzcd}
            D^{-1} \BPGm_{hH} \arrow[r] \arrow[d] & D^{-1} \BPGm^{hH} \arrow[d]
            \\ \left(L_{K(h)}D^{-1} \BPGm \right)_{hH} \arrow[r] & \left(L_{K(h)} D^{-1} \BPGm \right)^{hH}
        \end{tikzcd}
    \end{center}
    of spectra. Since $K(h)$-local spectra are closed under limits, the bottom right corner is $K(h)$-local. Applying $L_{K(h)}$ to the square thus yields a commutative square
    \begin{center}
        \begin{tikzcd}
            L_{K(h)} \left(D^{-1} \BPGm_{hH} \right) \arrow[r] \arrow[d] & L_{K(h)} \left(D^{-1} \BPGm^{hH} \right) \arrow[d]
            \\ \left(L_{K(h)}D^{-1} \BPGm \right)_{hH} \arrow[r] & \left(L_{K(h)} D^{-1} \BPGm \right)^{hH}
        \end{tikzcd}
    \end{center}
    The left map is an equivalence since $L_{K(h)}$ commutes with colimits. The bottom map is an equivalence by $K(h)$-local Tate vanishing \cite{greenleesSadofsky1996tate, hoveySadofsky1996tate}. Moreover, $\left(L_{K(h)} D^{-1} \BPGm \right)^{hH}$ is non-zero by \cref{cor:fixed_points_EOn_equivalence}. So the top right corner of the square cannot be $0$.
\end{proof}

When we set out on this project, one of our original motivations was to improve \cref{cor:BPGmEh_equivalence} by constructing an equivariant equivalence 
\[
\ul{\map}_G \left(\Sigma_+^{\infty} \mathrm{EG}, L_{\Infl_e^G K(h)} \BPGm \right) \xrightarrow{\ \simeq \ } E(k,\Gamma_h)^{hC(k,m)}
\]
via the Beaudry--Hill--Shi--Zeng orientation (\cref{thm:BHSZ_refinement}). It would generalize the non-equivariant Baker--Würgler equivalence $L_{K(h)} \BP \langle h \rangle \simeq \widehat{E}(h)$, see \cref{remark: Baker wuergler}, and was expected by experts to be true.

We were surprised to find that it is actually not an equivalence in general. In fact, the failure can already be seen on the underlying non-equivariant level. Let us recall and set up some notation first.

\begin{recollection}
    \hfill \label{recollection: BHSZ computations}
    \begin{enumerate}[(i)]
        \item There exist certain elements $\tb{t_i^{C_{2^n}}} \in \pi_{2(2^i-1)}^e \BP^{( \! ( C_{2^n} ) \! )}$ for $i \geq 1$ determining the group action on a specified formal group law. See \cite[(1.3)]{beaudryHillShiZeng2021modelsLubinTate} for a precise definition. With these,
        \[ \pi_*^e \BP^{( \! ( C_{2^n} ) \! )} \cong \Z_{(2)}[C_{2^n} \cdot t_1^{C_{2^n}}, C_{2^n} \cdot t_2^{C_{2^n}}, \cdots] \]
        where $C_{2^n} \cdot x = \{x, \gamma_n x, \gamma_n^2 x, \cdots, \gamma_n^{2^{n-1}-1}x \}$ with $\gamma_n$ denoting a generator of $C_{2^n}$.
        \item We will write $\tb{I_k} \coloneqq (2,v_1, \cdots, v_{k-1}) \subseteq \pi_* \BP^{( \! ( C_{2^n}) \! )}$ where $v_i \in \pi_* \BP \to \pi_* \BP^{( \! ( C_{2^n}) \! )}$ are the Araki generators.
        \item There is a recursive formula
        \[ t_k^{C_{2^{n-1}}} \equiv t_k^{C_{2^n}} + \gamma_n t_k^{C_{2^n}} + \sum_{j=1}^{k-1} \gamma_n t_j^{C_{2^n}} (t_{k-j}^{C_{2^n}})^{2^j} \pmod{I_k} \]
        by \cite[Theorem 1.1]{beaudryHillShiZeng2021modelsLubinTate}. Here, $t_k^{C_2} \equiv v_k \pmod{I_k}$, see \cite[Proposition 3.5]{beaudryHillShiZeng2021modelsLubinTate}. We will later relate $v_k$ to $t_k^{C_{2^n}}$ for $n > 1$, see \cref{theorem: vk is Vk}.
    \end{enumerate}
\end{recollection}

\begin{remark} \label{remark: coefficients composition}
    At this point, we want to give a remark in pure algebra. Let $R$ be a ring of characteristic $2$. If $A(T) = T + \sum_{k \geq 1} a_k T^{2^k} \in R \llbracket T \rrbracket$ and $B(T) = T + \sum_{k \geq 1} b_k T^{2^k} \in R \llbracket T \rrbracket$, then the coefficients of the composition $B(A(T)) = T + \sum_{k \geq 1} c_k T^{2^k}$ are given by 
    \[ c_k = a_k + b_k + \sum_{j=1}^{k-1} b_{j}a_{k-j}^{2^j}. \] This is precisely the sort of formula appearing in Beaudry--Hill--Shi--Zeng's recursive formula. We will use this observation later in \cref{theorem: vk is Vk}.
\end{remark}

Let us furthermore recall and record the following formula for $K(h)$-localization, since it will be important for our computations.

\begin{remark} \label{remark: Kh localization}
    Let $E$ be a complex oriented spectrum such that the image of $(2,v_1, \cdots, v_{h-1})$ under $\MU \to E$ forms a regular sequence in $\pi_*E$. Then,
    \[ \pi_*(L_{K(h)}E) \cong v_h^{-1}\pi_*(E)_{(2,v_1,\cdots,v_{h-1})}^{\wedge}, \]
    see e.g.~\cite[Proposition 7.4]{beaudryHillShiZeng2021modelsLubinTate}. Examples are $\Res_e^G D^{-1} \BPGm$ by \cite[Proposition 7.4]{beaudryHillShiZeng2021modelsLubinTate} and $\Res_e^G \BPGm$ by \cite[Corollary 3.3]{carrick2025higherrealktheoriesfinite}.
\end{remark}

Let us begin by collecting all positive examples. 

\begin{proposition}
    \hfill \label{prop: positive examples lubin tate}
    \begin{enumerate}[(i)]
        \item Let $h \geq -1$. The natural map 
        \[ \ul{\map}_{C_2}\left(\Sigma_+^{\infty} \mathrm{E}C_2, L_{\Infl_e^{C_2}K(h)} \BPR\langle h \rangle \right) \longrightarrow \ul{\map}_{C_2}\left(\Sigma_+^{\infty} \mathrm{E}C_2, L_{\Infl_e^{C_2}K(h)} D^{-1} \BPR\langle h \rangle \right) \]
        is an equivalence of $C_2$-spectra.
        \item The natural map 
        \[ \ul{\map}_{C_4}\left(\Sigma_+^{\infty} \mathrm{E}C_4, L_{\Infl_e^{C_4} K(2)} \BP^{( \! ( C_4 ) \! )} \langle 1 \rangle \right) \longrightarrow \ul{\map}_{C_4}\left(\Sigma_+^{\infty} \mathrm{E}C_4, L_{\Infl_e^{C_4} K(2)} D^{-1}\BP^{( \! ( C_4 ) \! )} \langle 1 \rangle \right) \] 
        is an equivalence of $C_4$-spectra.
    \end{enumerate}
\end{proposition}

\begin{proof}
    Equivalences of Borel spectra can be checked on underlying. We will use the $L_{K(h)}$ formula from \cref{remark: Kh localization}.
    \begin{enumerate}[(i)]
        \item We fully understand $\pi_*^{C_2} E(k,\Gamma_h)$ since $E(k, \Gamma_h)$ is strongly even \cite[Theorem 1.9]{hahnRealOrientationsLubin2020}. With this, we can deduce $D = \ol{v}_h$ using \cite[Proposition 6.3]{beaudryHillShiZeng2021modelsLubinTate}, so the desired follows immediately.
        \item First postcompose by the equivalence 
        \[ \ul{\map}_{C_4}\left(\Sigma_+^{\infty} \mathrm{E}C_4, L_{\Infl_e^{C_4} K(2)} D^{-1}\BP^{( \! ( C_4 ) \! )} \langle 1 \rangle \right) \xrightarrow{\ \simeq \ } E(k,\Gamma_2)^{hC(k,1)}, \]
        so it suffices to show that the composite 
        \[ L_{K(2)} \Res_e^{C_4} \BP^{( \! ( C_4 ) \! )} \langle 1 \rangle \longrightarrow E(k,\Gamma_2)^{hC(k,1)} \]
        is an equivalence. By \cite[Proposition 5.4]{beaudryHillShiZeng2021modelsLubinTate} this amounts to the inclusion map
        \[ \Z_{(2)} \left[t_1^{C_4}, \gamma t_1^{C_4} \right][v_2^{-1}]_{(2,v_1)}^{\wedge} \longrightarrow \Z_{2} \left[t_1^{C_4}, \gamma t_1^{C_4}, (t_1^{C_4})^{-1}, (\gamma t_1^{C_4})^{-1} \right]_{(t_1^{C_4} - \gamma t_1^{C_4}, t_1^{C_4} + \gamma t_1^{C_4})}^{\wedge}. \]
        On the other hand, $(t_1^{C_4} - \gamma t_1^{C_4}, t_1^{C_4} + \gamma t_1^{C_4}) = (2, t_1^{C_4}+\gamma t_1^{C_4})$. Using \cref{recollection: BHSZ computations} we get
        \begin{align*}
            v_1 &\equiv t_1^{C_4} + \gamma t_1^{C_4} \pmod 2,
            \\ v_2 &\equiv (\gamma t_1^{C_4})(t_1^{C_4})^{2} \pmod{(2,v_1)},
        \end{align*}
        so we see that the map is an isomorphism. \qedhere
    \end{enumerate}
\end{proof}

Certainly, \cref{prop: positive examples lubin tate}(i) was already known by Beaudry--Hill--Shi--Zeng, although it was not explicitly spelled out. Applying \cref{cor:BPGmEh_equivalence} we deduce:

\begin{corollary}
    \hfill 
    \begin{enumerate}[(i)]
        \item Suppose that $k^{\times}$ contains all $(2^h-1)$-roots of unity. There is an equivalence 
        \[ E(k,\Gamma_h)^{hC(k,h)} \simeq \ul{\map}_{C_2}\left(\Sigma_+^{\infty} \mathrm{E}C_2, L_{\Infl_e^{C_2}K(h)} \BPR\langle h \rangle \right) \] 
        of $C_2$-spectra.
        \item There is an equivalence $E(k, \Gamma_2)^{hC(k,1)} \simeq \ul{\map}_{C_4}\left(\Sigma_+^{\infty} \mathrm{E}C_4, L_{\Infl_e^{C_4} K(2)} \BP^{( \! ( C_4 ) \! )} \langle 1 \rangle \right)$ 
    of $C_4$-spectra.
    \end{enumerate}
\end{corollary}

\begin{remark} \label{remark: Baker wuergler}
    In particular, $E(k, \Gamma_h)^{hC(k,h)}$ is a $C_2$-equivariant refinement of completed Johnson--Wilson theory $L_{K(h)}\BP\langle h \rangle \simeq \widehat{E}(h)$ by a theorem of Baker--Würgler \cite{bakerWuergler1989liftings}.
\end{remark}

\begin{remark}
    In forthcoming work \cite{Qtmf15},
    the first-named author produces a \(C_4\)-equivariant equivalence \(\tmf_1(5)\simeq \BP^{(\!(C_4)\!)}\langle 1\rangle\) that holds before applying \(K(h)\)-localization. This constructs a \(\E^{C_4}_\infty\)-ring structure on \(\BP^{(\!(C_4)\!)}\langle 1\rangle\), and hence an \(\E_\infty\)-ring structure on \(\BP^{(\!(C_4)\!)}\langle 1\rangle^{C_4}\).
    This shows that \(\tmf_1(5)\) is a \(C_4\)-analogue of \(\kuR\) \cite{HillHopkinsRavenel2017C4RealKTheory}.
\end{remark}

Next, we will embark on proving that all other examples are negative examples. Let us begin with $L_{K(4)} \BP^{( \! ( C_4 ) \! )} \langle 2 \rangle$, where we are able to provide an explicit computation. 

\begin{lemma} \label{lemma: BP C4 2 mod I computation}
    There is an isomorphism
    \[ \pi_*^e\left(L_{K(4)}\BP^{( \! ( C_4 ) \! )}\langle 2 \rangle \right)/(2,v_1, v_2, v_3) \cong \F_4 \left[(t_1^{C_4})^{\pm 1} \right] \times \F_2 \left[(t_2^{C_4})^{\pm 1} \right] \]
    of $\BP_*$-modules.
\end{lemma}

\begin{proof}
    We use \cref{remark: Kh localization} to see
    \[
    \pi_*\left(L_{K(4)}\Res_e^{C_4}\BP^{( \! ( C_4 ) \! )}\langle 2 \rangle \right)\cong \left(v_4^{-1}\pi_*(\Res_e^{C_4}\BP^{( \! ( C_4 ) \! )}\langle 2 \rangle)\right)^{\wedge}_{I_4}.
    \]
    Using \cite[Theorem 1.1]{beaudryHillShiZeng2021modelsLubinTate} and \cite[Proposition 3.5]{beaudryHillShiZeng2021modelsLubinTate}, we have the following explicit formulae for the \(v_i\)'s in terms of the \(t_i^{C_4}\)'s:
    \begin{align*}
        v_1 &\equiv t_1^{C_4}+\gamma_2 t_1^{C_4} &&\operatorname{mod} \,\, (2)\\
        v_2 &\equiv t_2^{C_4}+\gamma_2 t_2^{C_4}+\gamma_2 t_1^{C_4}\left(t_{1}^{C_4}\right)^{2} &&\operatorname{mod} \,\,  (2,v_1)\\
        v_3 &\equiv \gamma_2 t_1^{C_4}\left(t_{2}^{C_4}\right)^{2} +\gamma_2 t_2^{C_4}\left(t_{1}^{C_4}\right)^{4}    && \operatorname{mod} \,\, (2,v_1,v_2)\\
        v_4 &\equiv \gamma_2t_2^{C_4}\left(t_{2}^{C_4}\right)^{4}   &&\operatorname{mod}\,\,  (2,v_1,v_2,v_3).
    \end{align*}
    Therefore, 
    \[I_4 =  \left(
    2,
    \,\,
    t_1^{C_4}+\gamma_2 t_1^{C_4},
    t_2^{C_4}+\gamma_2 t_2^{C_4}+\gamma_2 t_1^{C_4}\left(t_{1}^{C_4}\right)^{2},
    \,\,
    \gamma_2 t_1^{C_4}\left(t_{2}^{C_4}\right)^{2} +\gamma_2 t_2^{C_4}\left(t_{1}^{C_4}\right)^{4}
    \right)
    \]
    Using \(\pi_*^e(\BP^{( \! ( C_4 ) \! )}\langle 2 \rangle)\cong \Z_{(2)}[t_1^{C_4},\gamma_2 t_1^{C_4},t_2^{C_4},\gamma t_2^{C_4}]\), we can write down an isomorphism
    \[
    \pi_*^e \left(\BP^{( \! ( C_4 ) \! )}\langle 2 \rangle \right)/I_4 \xrightarrow{\ \sim \ } 
    \F_2[t_1^{C_4},t_2^{C_4}]/
    \left(
    t_1^{C_4}\left(
    (t_2^{C_4})^2 
    +
    t_2^{C_4}(t_1^{C_4})^3
    +
    (t_1^{C_4})^6
    \right)
    \right)
    \]
    induced by
    \[
    t_1^{C_4} \mapsto t_1^{C_4}, \quad 
    \gamma_2 t_1^{C_4} \mapsto t_1^{C_4},\quad 
    t_2^{C_4} \mapsto t_2^{C_4},\quad 
    \gamma_2 t_2^{C_4}\mapsto t_2^{C_4} + (t_1^{C_4})^3.
    \]
    Furthermore, since \(v_4 \equiv \gamma_2t_2^{C_4}\left(t_{2}^{C_4}\right)^{4}\) modulo \(I_4\), we learn that
    \[
    \left(v_4^{-1}\pi_*^e(\BP^{( \! ( C_4 ) \! )}\langle 2 \rangle)\right)^{\wedge}_{I_4}/I_4 \cong
    \F_2 \left[t_1^{C_4},t_2^{C_4},
    (t_2^{C_4})^{-1},
    \left((t_1^{C_4})^3+t_2^{C_4} \right)^{-1}
    \right]/
    \left(
    t_1^{C_4}\left(
    (t_2^{C_4})^2 
    +
    t_2^{C_4}(t_1^{C_4})^3
    +
    (t_1^{C_4})^6
    \right)
    \right)
    \]
    By the Chinese remainder theorem this splits into 
    \[ \F_2 \left[(t_2^{C_4})^{\pm 1} \right] \times \F_2\left[(t_1^{C_4})^{\pm 1}, T \right]/(T^2 + T + 1) \cong \F_2 \left[(t_2^{C_4})^{\pm 1} \right] \times \F_4 \left[(t_1^{C_4})^{\pm 1} \right], \]
    where \(T=t_2^{C_4}/(t_1^{C_4})^3\).
\end{proof}

\begin{corollary}
    There is no equivalence between \(L_{K(4)} \Res^{C_4}_e\BP^{(\!(C_4)\!)}\langle 2\rangle\) and \(L_{K(4)} \Res^{C_4}_e D^{-1}\BP^{(\!(C_4)\!)}\langle 2\rangle\) as $\BP$-modules.
\end{corollary}

\begin{proof}
    For brevity, write \(X=\Res^{C_4}_e\BP^{(\!(C_4)\!)}\langle 2\rangle\), and \(Y=\Res^{C_4}_e D^{-1}\BP^{(\!(C_4)\!)}\langle 2\rangle\).
    By \cite[Proposition 7.1]{beaudryHillShiZeng2021modelsLubinTate} we have
    \[
    \pi_* \left(L_{K(4)}\Res^{C_4}_e D^{-1}\BP^{(\!(C_4)\!)}\langle 2 \rangle \right)/(2,v_1,v_2,v_3)\cong \F_2[(t^{C_4}_2)^{\pm 1}].
    \]
    On the other hand, 
    \[ \pi_*\left(L_{K(4)}\Res_e^{C_4}\BP^{( \! ( C_4 ) \! )}\langle 2 \rangle \right)/(2,v_1, v_2, v_3) \cong \F_4 \left[(t_1^{C_4})^{\pm 1} \right] \times \F_2 \left[(t_2^{C_4})^{\pm 1} \right] \] 
    by \cref{lemma: BP C4 2 mod I computation}.
\end{proof}

One of the many ways to distinguish $\F_2[(t_2^{C_4})^{\pm 1}]$ from $\F_4 \left[(t_1^{C_4})^{\pm 1} \right] \times \F_2 \left[(t_2^{C_4})^{\pm 1} \right]$ is by counting the $\F_4$-points on the corresponding affine schemes, i.e.~by counting the number of ring homomorphisms to $\F_4$. While there are three ring homomorphisms
\[ \F_2 \left[(t_2^{C_4})^{\pm 1} \right] \longrightarrow \F_4, \] the ring $\F_4 \left[(t_1^{C_4})^{\pm 1} \right] \times \F_2 \left[(t_2^{C_4})^{\pm 1} \right]$ has nine $\F_4$-points.

We will now greatly generalize this observation to deal with all other possibilities for $L_{K(h)} \Res_e^{G} \BPGm$. The first step is to generalize Beaudry--Hill--Shi--Zeng's recursive formula (\cref{recollection: BHSZ computations}) for inductively describing the generators of \(\pi_*^e(\BPG)\). 

\begin{notation}
    Let $n \geq 1$.
    \begin{enumerate}[(i)]
        \item Let $r \geq 0$. Then, we write $\tb{P_r(T)} \coloneqq T + \sum_{i \geq 1} \gamma_n^r t_i^{C_{2^n}} T^{2^i} \in \pi_*^e(\BP^{( \! ( C_{2^n} ) \! )}) \llbracket T \rrbracket$.
        \item Let $\tb{P(T)} \coloneqq P_{2^{n-1}-1} \circ P_{2^{n-1}-2} \circ \cdots \circ P_0(T) \in \pi_*^e(\BP^{( \! ( C_{2^n}) \! )}) \llbracket T \rrbracket$. It is then of the form
        \[ P(T) = T + \sum_{k \geq 1} V_k T^{2^k} \]
        for some $\tb{V_k} \in \pi_*^e(\BP^{( \! ( C_{2^n}) \! )})$.
    \end{enumerate}
\end{notation}

We suggestively write $V_k$ due to the following result.

\begin{theorem} \label{theorem: vk is Vk}
    Let $n \geq 1$. In the above notation, we have \(v_k\equiv V_k \pmod{I_k}\). 
\end{theorem}
\begin{proof}
    For \(0\leq s\leq n-1\), and \(0\leq b\leq 2^{n-1-s}-1\)
    consider the power series
    \[
    Q_{s,b}(T)\coloneqq P_{(b+1)2^s-1}\circ P_{(b+1)2^s-2}\circ \cdots   \circ P_{b2^s}(T)= T+\sum_{k\geq 1}Q_{s,b,k}T^{2^k} \in \pi_*^e(\BP^{( \! ( C_{2^n}) \! )}) \llbracket T \rrbracket.
    \]
    for some $Q_{s,b,k} \in \pi_*^e(\BP^{( \! ( C_{2^n} ) \! )})$. We prove the following statement by induction on $s$: 
    For every \(k\geq 1\), we have
    \[
    Q_{s,b,k}\equiv \gamma_n^{b2^s}t_k^{C_{2^{n-s}}} \pmod{I_k}.
    \]
    Let us first explain that the theorem follows from this claim: By definition, \(P(T)=Q_{n-1,0}(T)\). So the claim gives \(V_k\equiv t_k^{C_2} \pmod{I_k}\). Moreover, \(v_k\equiv t_k^{C_2} \pmod{I_k} \) by \cref{recollection: BHSZ computations}. Hence, we deduce \(v_k\equiv V_k \pmod{I_k}\). 

    For $s = 0$ we have $Q_{0,b}(T) = P_b(T)$ which is already of this form. Suppose that the claim holds for some \(s\) and fix \(b\) such that \(0\leq b \leq 2^{n-2-s}-1\).
    By construction, we have \(Q_{s+1,b}(T)=Q_{s,2b+1}\circ Q_{s,2b}(T)\).
    Looking at the coefficient of \(T^{2^k}\) in \(Q_{s+1,b}(T)\) with \cref{remark: coefficients composition}, we learn that
    \[
    Q_{s+1,b,k}
    \equiv
    Q_{s,2b,k}
    +
    Q_{s,2b+1,k}
    +
    \sum_{j=1}^{k-1}
    Q_{s,2b+1,j}
    \left(Q_{s,2b,k-j}\right)^{2^j} \pmod{2}.
    \]
    Using the formula for \(Q_{s,2b,k}\) and \(Q_{s,2b+1,k}\) from the inductive hypothesis and using $\gamma_{n-s} = \gamma_n^{2^s}$, we learn that
    \begin{align*}
        Q_{s+1,b,k} &\equiv
    \gamma_n^{b2^{s+1}}t_k^{C_{2^{n-s}}}
    +
    \gamma_n^{b2^{s+1}}\gamma_{n-s}t_k^{C_{2^{n-s}}}
    +
    \sum_{j=1}^{k-1}
    \left(\gamma_n^{b2^{s+1}}\gamma_{n-s}t_j^{C_{2^{n-s}}}\right)
    \left(\gamma_n^{b2^{s+1}}t_{k-j}^{C_{2^{n-s}}}\right)^{2^j}
    \pmod{I_k}.
    \\ &\equiv \gamma_n^{b2^{s+1}} \left( t_k^{C_{2^n-s}} + \gamma_{n-s} t_k^{C_{2^{n-s}}} + \sum_{j=1}^{k-1}
    \left(\gamma_{n-s}t_j^{C_{2^{n-s}}}\right)
    \left(t_{k-j}^{C_{2^{n-s}}}\right)^{2^j} \right) \pmod{I_k}
    \end{align*}
    Inside the brackets, we can now apply Beaudry--Hill--Shi--Zeng's formula (\cref{recollection: BHSZ computations}). Together with the invariance of \(I_k\) under the action \cite[Proposition 3.7]{beaudryHillShiZeng2021modelsLubinTate}, we conclude
    \[
    Q_{s+1,b,k}\equiv
    \gamma_n^{b2^{s+1}}t_k^{C_{2^{n-(s+1)}}} \pmod{I_k}.
    \]
    This completes the proof.
\end{proof}
Now we use this to count \(\mathbb{F}_4\)-points. By \cite[Proposition 7.1]{beaudryHillShiZeng2021modelsLubinTate}
    \[
    \pi_*^e \left(L_{K(h)}D^{-1}\BP^{(\!(G)\!)\langle m \rangle} \right)/I_h \cong \F_2[(t^{G}_m)^{\pm 1}].
    \]
and it only has three $\F_4$-points. We want to prove that \(\pi^e_*(L_{K(h)}\BPGm)/I_h\) has more $\F_4$-points. Let us first set up some notation. 
\begin{observation}
    For $i \geq 1$ and $r \geq 0$ let us write \(\tb{a_{r,i}} \coloneqq \gamma_{n}^rt_{i}^{G} \in \pi_*^e(\BP^{( \! ( C_{2^n} ) \! )}) \), so that we have $P_r(T) = T + \sum_{i \geq 1} a_{r,i} T^{2^i}$. By \cref{remark: Kh localization} and \cref{theorem: vk is Vk} we infer
    \[ \pi^e_* \left(L_{K(h)}\BPGm \right)/I_h
        \cong
    \mathbb F_2[a_{r,i}:0\leq r\leq 2^{n-1}-1, 1\leq i\leq m][V_h^{-1}]/(V_1,\ldots,V_{h-1}) \]
    and we are interested in its $\F_4$-points, i.e.~its ring homomorphisms to $\F_4$. In the language of classical algebra we have polynomials $V_1, \cdots, V_h \in \F_2[a_{r,i}]_{r,i}$ and we wish to choose an $\alpha_{r,i} \in \F_4$ for each $a_{r,i}$ such that 
    \[ V_1((\alpha_{r,i})_{r,i}) = V_2((\alpha_{r,i})_{r,i}) = \cdots = V_{h-1}((\alpha_{r,i})_{r,i}) = 0 \quad \text{and} \quad V_h((\alpha_{r,i})_{r,i}) \neq 0 \]
    in $\F_4$.
\end{observation}

\begin{notation}
    Consider a collection of elements \(\alpha=\{\alpha_{r,i}\in \F_4:0\leq r\leq 2^{n-1}-1, \ 1\leq i\leq m\}\).
    \begin{enumerate}[(i)]
        \item Let $r \geq 0$. We write $
        \tb{P^\alpha_{r}(T)} \coloneqq T+\sum_{i = 1}^m\alpha_{r,i} T^{2^{i}} \in \F_4 [T ]$.
        \item Let $\tb{P^{\alpha}(T)} \coloneqq  P^\alpha_{2^{n-1}-1}\circ P_{2^{n-1}-2}^{\alpha} \circ\cdots\circ P^\alpha_0(T) \in \F_4 [T]$. So
        \[ P^\alpha(T)=T+\sum_{k}V^\alpha_k T^{2^{k}}. \]
        for certain $V_k^{\alpha} \in \F_4$ described by polynomials in $\alpha_{r,i}$.
    \end{enumerate}
    So an \(\F_4\)-point of \(\pi^e_*(L_{K(h)}\BPGm)/I_h\) is a choice of elements \(\alpha\) such that \(V^\alpha_1=V^\alpha_2=\cdots =V^\alpha_{h-1}=0\) and \(V^\alpha_h\neq 0\). In other words, an $\F_4$-point is an $\alpha$ such that $P^{\alpha}(T) = T + V_h^{\alpha} T^{2^h}$ with $V_h^{\alpha} \neq 0$.
\end{notation}

\begin{proposition} \label{prop: F4 points computation}
    Let $G = C_{2^n}$ and $h = 2^{n-1}m$.
    \begin{enumerate}[(i)]
        \item Let $n \geq 3$ and $m = 1$. Then, the number of $\F_4$-points of \(\pi_*^e(L_{K(h)}\BPG\langle 1\rangle)/I_h\) is at least $3^{2^{n-2}}$.
        \item Let $n \geq 2$ and $m \geq 2$. Then, \(\pi^e_*(L_{K(h)}\BPGm)/I_h\) has at least four $\F_4$-points.
    \end{enumerate}
\end{proposition}

\begin{proof}
    This is a pure algebra problem. Namely, we will look for $\F_4$-points $\alpha$, i.e.~choices of elements such that $P^{\alpha}(T) = T + V_h^{\alpha} T^{2^h}$ with $V_h^{\alpha} \neq 0$.
    \begin{enumerate}[(i)]
        \item Let $(c_0, c_1, \cdots, c_{2^{n-2}-1})$ be a tuple in $\F_4^{\times}$, so in particular $c_q^3 = 1$. We then pick
        \[ (\alpha_{0,1}, \alpha_{1,1}, \cdots, \alpha_{2^{n-1}-1,1}) = (c_0, c_0, c_1, c_1, \cdots, c_{2^{n-2}-1}, c_{2^{n-2}-1}). \]
        In this case, $P_{2q}^{\alpha}(T) = P_{2q+1}^{\alpha}(T) = T + c_q T^2$ for $q \geq 0$, so 
        \[ P_{2q+1}^{\alpha} \circ P_{2q}^{\alpha}(T) = T + c_q T^2 + c_q T^2 + c_q^3 T^4 = T + T^4. \]
        Iterating this procedure yields
        \[ T+V^\alpha_1T^2+\cdots+V^\alpha_hT^{2^h} = P^\alpha_{2^{n-1}-1}\circ\cdots\circ P^\alpha_0(T) =T+T^{2^h}. \]
        Thus, we have constructed $3^{2^{n-2}}$ examples of $\F_4$-points.
        \item We will construct polynomials
        \[ A_m(T) = T + \sum_{k=1}^{m} p_k T^{2^k} \in \F_4[T] \quad \text{and} \quad B_m(T) = T + \sum_{k = 1}^m q_k T^{2^k} \in \F_4[T] \]
        such that $B_m(A_m(T)) = T + T^{2^{2m}} $. Then, choosing
        \[ P_0^{\alpha} = A_m, \ P_1^{\alpha} = B_m, \ P_{q}^{\alpha} = T + T^{2^m} \]
        for $q \geq 2$ yields
        \[ P^{\alpha}(T) = P_{2^{n-1}-1}^{\alpha} \circ P_{2^{n-1}-2}^{\alpha} \circ \cdots \circ P_0^{\alpha}(T) = (T + T^{2^{2m}})^{\circ 2^{n-2}} = T + T^{2^h}. \]
        The choices of $p_k$ and $q_k$ are thus in particular $\F_4$-points of the desired form. Three such points are obtained by setting $p_k = q_k = 0$ for $k \leq m-1$ and $p_m = q_m \in \F_4^{\times}$. This is by the same computation as in (i). We need to construct one more. Let $\omega \in \F_4$ be such that $\omega^2 + \omega + 1 = 0$. We put 
        \[
            A_m(T) = T + \omega^2 \sum_{k=1}^{m-1} T^{2^k} + \omega T^{2^m} \quad \text{and} \quad B_m(T) = T + \sum_{k=1}^{m-1} q_k T^{2^k} + q_m T^{2^m} \]
        with
        \[ q_k = \begin{cases}
            \omega^2 \quad & k<m \text{ odd},
            \\ \omega & k<m \text{ even}
        \end{cases} \qquad \text{and} \qquad q_m = \begin{cases}
            \omega \quad & m \text{ odd},
            \\ \omega^2 & m \text{ even}.
        \end{cases} \]
        Let us now outline that these fulfill the desired properties.
        \[ S_m(T) = \sum_{k=0}^{m-1} T^{2^k} \quad \text{and} \quad R(T) = T + \omega T^2. \]
        One may compute $A_m(T) = R(S_m(T))$ and $B_m(T) = S_m(R(T))$. It is then an algebra exercise to verify that
        \[ B_m(A_m(T)) = S_m(R(R(S_m(T)))) = T + T^{2^{2m}}. \]
        \qedhere
    \end{enumerate}
\end{proof}

By \cite[Proposition 7.1]{beaudryHillShiZeng2021modelsLubinTate}
    \[
    \pi_*^e \left(L_{K(h)} D^{-1}\BP^{(\!(G)\!) \langle m \rangle} \right)/I_h \cong \F_2[(t^{G}_m)^{\pm 1}].
    \]
and it only has three $\F_4$-points, so using \cref{prop: F4 points computation} we deduce: 
\begin{theorem} \label{theorem: not higher real k theory}
    Let \(G=C_{2^n}\), and \(h=2^{n-1}m\). Suppose \(n\geq 3, m\geq 1\) or \(n\geq 2,m\geq 2\).
    Then, there is no equivalence between \(L_{K(h)}\Res^G_e\BPGm\) and \(L_{K(h)}\Res^G_e D^{-1} \BPGm\) of $\BP$-modules.
\end{theorem}

\subsection{Periodicity for Lubin--Tate theories}
\label{subsection: periodicity lubin tate}
To the best of the authors' knowledge, \cref{thm:BHSZ_refinement} gives the first construction of \(G\)-equivariant maps 
\[ \BPGm\longrightarrow E(k,\Gamma_h) \] 
factoring the Beaudry--Hill--Shi--Zeng \(\BPG\)-orientation of \(E(k,\Gamma_h)\) for all \(G=C_{2^n}\) and \(h=2^{n-1}m\).
The existence of these maps has already been used to great effect in the literature, see \cite{meier2024transchromaticphenomenaequivariantslice,duan2025periodicityfinitecomplexityhigher}.

In this section, we briefly summarize various applications of these maps in the literature, particularly to determine periodicities for \(E(k,\Gamma_h)\). We aim to be brief and will implicitly adopt the notation and terminology from \cite{meier2024transchromaticphenomenaequivariantslice,duan2025periodicityfinitecomplexityhigher}, which the interested reader may look into. The goal of this section is only to highlight instances where these maps have been used. We thank Lennart Meier and XiaoLin Danny Shi for helpful conversations.

Hill--Hopkins--Ravenel conjectured\footnote{See {\cite[Conjecture 1.1]{meier2024transchromaticphenomenaequivariantslice}} and the surrounding discussion for a written account of this conjecture.} that the slice spectral sequences computing \(E(k,\Gamma_h)^{hG}\) relate to each other for varying heights \(h\) and groups \(G\). Inspired by this, work of Meier--Shi--Zeng showed that  the slice spectral sequences for \(\BPGm\) exhibit such transchromatic phenomena.
\begin{theorem}[Transchromatic Isomorphism for \(\BPGm\), {\cite[Theorem A]{meier2024transchromaticphenomenaequivariantslice}}]
    Let \(G=C_{2^{n+1}}\) and \(m\geq 1\). 
    There is a shearing isomorphism $d_{2r-1} \leftrightsquigarrow d_r$ between the following regions of spectral sequences: 
\begin{enumerate}[(i)]
\item The $G$-slice spectral sequence of $\BPGm$ on or above the line of slope $1$; and  
\item The $(G/C_2)$-slice spectral sequence of $\BP^{(\!(G/C_2)\!)} \langle m \rangle$.
\end{enumerate}
\end{theorem}
Using the maps \(\BPGm\to E(k,\Gamma_h)\) supplied by \cref{thm:BHSZ_refinement}, naturality of the slice spectral sequence then gives an explicit correspondence between the differentials in the \(C_{2^{n+1}}\)-slice spectral sequence of \(E(k,\Gamma_h)\) in a region with the differentials in the \(C_{2^{n}}\)-slice spectral sequence of \(E(k,\Gamma_{h/2})\). In particular, Meier--Shi--Zeng deduce a transchromatic isomorphism and periodicity theorem for Lubin--Tate theories.

\begin{theorem}[Transchromatic Isomorphism for \(E(k,\Gamma_h)\), {\cite[Theorem 8.3]{meier2024transchromaticphenomenaequivariantslice}}]
    Let \(V\in \RO(G/C_2)\). 
    The class \(u_V\) is a permanent cycle in the slice spectral sequence of \(E(k,\Gamma_h)\) if and only if the class \(u_V\) is a permanent cycle in the slice spectral sequence of \(E(k,\Gamma_{h/2})\).
\end{theorem}

Using work by Hu--Kriz \cite{HuKrizReal}, Hill--Hopkins--Ravenel \cite{HillHopkinsRavenel2017C4RealKTheory}, and Hill--Shi--Wang--Xu \cite{hill2023slice} compute periodicities for \(\BPRn\), \(\BP^{(\!(C_4)\!)}\langle 1\rangle\), and \(\BP^{(\!(C_4)\!)}\langle 2\rangle\) respectively. Meier--Shi--Zeng use the Transchromatic Isomorphism to deduce a number of periodicities, and import them via the maps \(\BPGm\to E(k,\Gamma_h)\) to deduce periodicities for the Lubin--Tate theories.

\begin{theorem}[Periodicity for \(E(k,\Gamma_h)\), {\cite[Theorem B]{meier2024transchromaticphenomenaequivariantslice}}]
    Let \(G=C_{2^{n+1}}\), \(m\geq 1\), and \(h=2^nm\) and \(V\in \RO(G/C_2)\).
    Then, \((\vert V\vert -V)\) is an \(\RO(G)\)-graded periodicity for \(E(k,\Gamma_h)\) if and only if \((\vert V\vert -V)\) is an \(\RO(G/C_2)\)-graded periodicity for \(E(k,\Gamma_{h/2})\).
\end{theorem}

More recently, Duan--Hill--Li--Liu--Shi--Wang--Xu determine a large class of \(\RO(G)\)-graded periodicities for \(E(k,\Gamma_h)\).
In particular, Meier--Shi--Zeng's Transchromatic Isomorphism theorem for \(E(k,\Gamma_h)\) plays a key role in part of {\cite[Part (1) of Theorem B]{duan2025periodicityfinitecomplexityhigher}}.

\begin{theorem}[Periodicity for \(E(k,\Gamma_h)\), {\cite[Part (1) of Theorem B]{duan2025periodicityfinitecomplexityhigher}}]
    Let \(G=C_{2^n}\), and \(h=2^{n-1}m\). 
    Let \(\rho_G\) denote the regular representation of \(G\). 
    Let \(\lambda_i\) denote the \(2\)-dimensional $G$-representation corresponding to rotation by \(\frac{\pi}{2^i}\) for \(0 \leq i \leq n-1\).
    The \(G\)-spectrum \(E(k,\Gamma_h)\) has the following \(\RO(G)\)-graded periodicities:
    \begin{enumerate}[(i)]
        \item \(\rho_{G}\), the regular representation of \(G\);
        \item \(2^{2^{n-i}m+n-i+1}-2^{2^{n-i}m+n-i}\lambda_{n-i}\), \(1\leq i\leq n\).
    \end{enumerate}
\end{theorem}

These are by no means all known applications of the orientation $\BPGm \to E(k,\Gamma_h)$, but we hope to have demonstrated the relevance of the map and in particular the strength of our structured orientations to be able to produce it.

\appendix

\section{Twisted monoid quotients via equivariant cubes}
\label{section: twisted monoid quotients via parametrized colimits}
Our discussion of twisted monoid quotients was fully in the language of Hill--Hopkins--Ravenel \cite{HHR16}, but we have also used the language of parametrized higher category theory. In the spirit of this, we record a universal property of twisted monoid quotients in terms of equivariant cubes. 

Non-equivariantly, quotients by multiple elements can be encoded as certain total cofibers of cubes. Twisted monoid quotients (\cref{construction:TMR2}) modify this construction by putting suitable actions on the cube. We make it precise in this appendix. Let us first define the relevant equivariant cubes.

\begin{construction}[{\cite[Construction 3.2.3, Notation 3.2.7]{hilman2024parametrisednoncommutativemotivesequivariant}}]
    Let $H \leq G$ be a subgroup.
    \begin{enumerate}[(i)]
        \item Let $\ul{\Delta}^1$ denote the constant $H$-category with value $\Delta^1$. Then, the $G$-$\infty$-category $\Coind_H^G \ul{\Delta}^1$ is a cube in each level, and vertices are given by strings of $0$'s and $1$'s of a suitable length. 
        We denote by $\tb{\ul{J}_H^G} \subseteq \Coind_H^G \ul{\Delta}^1$ the $G$-full subcategory, which on each level removes the vertex $1\cdots1$.
        \item Let $\ul{\C}^{\otimes}$ be a $G$-symmetric monoidal $\infty$-category and $(A \to B) = \phi \colon \ul{\Delta}^1 \to \Res_H^G \ul{\C}$ be a map in $\Res_H^G \ul{\C}$. Then, we write
        \[ \tb{\underset{\underline{J}_H^G}{\ul{\colim}} \ N_H^G (A \to B)} \coloneqq \ul{\colim} \left( \ul{J}_H^G \subseteq \Coind_H^G \ul{\Delta}^1 \xrightarrow{\Coind_H^G \phi} \Coind_H^G \Res_H^G \ul{\C} \xrightarrow{N_H^G} \ul{\C} \right).  \]
    \end{enumerate}
\end{construction}

Let $H \leq G$ and let $R$ be an $\E_{\infty}^G$-ring spectrum. Recall that the norm in $\ul{\LMod}_R(\ul{\Sp}^G)^{\otimes}$ is given by $N_R^{H \to G} M \simeq R \otimes_{N_H^G R} N_H^G M$, see e.g.~\cite[Example 3.3.14]{quinnZhu2026multiplicativeequivariantthomspectra}.

\begin{proposition}
    Let $H \leq G$ and let \(R\) be an \(\E^G_\infty\)-ring spectrum, \(A\) be an $\E_1$-\(R\)-algebra with a norm map $N_R^{H \to G} \Res_H^G A \to A$.\footnote{This is for example provided by a $\Coind_H^G \E_1$-$R$-algebra structure, see \cref{footnote: Coind}.} Let \(x\in \pi^H_V(R)\), then there is a cofiber sequence 
    \begin{center}
        \begin{tikzcd}
            \underset{\underline{J}_H^G}{\ul{\colim}} \left(N_R^{H \to G} \left(\Sigma^{V}\Res^G_H A{[x]} \to \Res^G_H A{[x]}\right)\right) \arrow[r] & A \arrow[r] & A/(G\cdot x).
        \end{tikzcd}
    \end{center}
    in $\LMod_R(\Sp^G)$.
\end{proposition}

\begin{proof}
    Consider the cofiber sequence \(\Sigma^{V} \Res_H^GR[x]\to \Res_H^G R[x]\to \Res_H^G R\) in $\LMod_R(\Sp^H)$.
    Since $\ul{\LMod}_R(\ul{\Sp}^G)^{\otimes}$ is a distributive symmetric monoidal $G$-$\infty$-category \cite[Theorem A]{quinnZhu2026multiplicativeequivariantthomspectra}, we can apply Hilman's result \cite[Proposition 3.2.8]{hilman2024parametrisednoncommutativemotivesequivariant}. According to it, applying $N_R^{H \to G}$ induces a cofiber sequence
    \begin{center}
        \begin{tikzcd}
            \underset{\underline{J}_H^G}{\ul{\colim}}\left(N^G_H\left(\Sigma^{V}{R[x]\to R[x]}\right)\right) \arrow[r] &  {R[G\cdot x]} \arrow[r] & R,
        \end{tikzcd}
    \end{center}
    Since base change preserves parametrized colimits \cite[Proposition 3.3.11]{quinnZhu2026multiplicativeequivariantthomspectra}, by base change along the map \(R[G\cdot x]\to N_R^{H \to G} A \to A\) we obtain the desired cofiber sequence.
\end{proof}

\begin{example}
    For $R = \S$ we obtain the cofiber sequence
    \begin{center}
        \begin{tikzcd}
            \underset{\underline{J}_H^G}{\ul{\colim}} \left(N_H^G \left(\Sigma^{V}\Res^G_H A{[x]} \to \Res^G_H A{[x]}\right)\right) \arrow[r] & A \arrow[r] & A/(G\cdot x).
        \end{tikzcd}
    \end{center}
    in $\Sp^G$.
\end{example}

\begin{remark}
    The proof relies on Hilman's \cite[Proposition 3.2.8]{hilman2024parametrisednoncommutativemotivesequivariant} which is written for \(G\)-symmetric monoidal categories.
    Examining Hilman's proof, it is enough to know that only the specific norm functor \(N^{H\to G}_{R}\) from \cite[Theorem A]{quinnZhu2026multiplicativeequivariantthomspectra} is distributive. In particular, the assumption that \(R\) is an \(\E^G_\infty\)-ring spectrum can be significantly weakened. 
\end{remark}

\bibliographystyle{alpha}
\bibliography{main}

\Addresses
\end{document}